\documentclass[a4paper]{article}

\usepackage[english]{babel}

\usepackage[utf8]{inputenc}
\usepackage[utf8]{inputenc}
\usepackage{amsmath}
\usepackage{graphicx}
\usepackage{amssymb}
\usepackage{amsthm}
\usepackage{tikz-cd}
\usepackage{mathrsfs}
\usepackage[colorinlistoftodos]{todonotes}
\usepackage{enumitem}
\usepackage{yfonts}
\usepackage{ dsfont }
\usepackage{MnSymbol}

\title{Local Nahm transform and singularity formation of ASD connections}

\author{Yang Li}

\date{\today}
\newtheorem{thm}{Theorem}[section]
\newtheorem{lem}[thm]{Lemma}

\theoremstyle{definition}
\newtheorem{defn}[thm]{Definition}
\newtheorem{eg}[thm]{Example}

\newtheorem{cor}[thm]{Corollary}

\newtheorem*{rmk}{Remark}
\newtheorem{prop}[thm]{Proposition}
\newtheorem{Problem}[thm]{Problem}
\newtheorem*{Acknowledgement}{Acknowledgement}

\newcommand{\ie}{\emph{i.e.} }
\newcommand{\cf}{\emph{cf.} }

\newcommand{\R}{\mathbb{R}}
\newcommand{\C}{\mathbb{C}}

\newcommand{\Lap}{\Delta}
\newcommand{\norm}[1]{\left\lVert #1 \right\rVert}

\DeclareMathOperator{\Hom}{Hom}
\DeclareMathOperator{\End}{End}

\begin{document}
\maketitle

%\section{ Some heuristic ideas   }

%We are faced with the problem of reconstructing the connection from the Dirac fields. There are roughly two possible strategies to proceed: 

%1. Use the finite dimensional space of Dirac fields having small eigenvalue, and come up with an approximate reconstruction.

%2. Use the information from the whole Hardy space and produce an exact reconstruction. Then argue that only a finite amount of information involved is truly relevant in the concentrated limit. So after collapsing the irrelevant information one should be able to produce an approximate reconstruction from a finite amount of data.

%There is a well known algorithm to reproduce the ASD connection on the 4-torus from information coming from Dirac fields, called the Nahm transform (sometimes also called the Fourier Mukai transform). One would hope that the second strategy can work by adapting the Nahm transform to the local setting. 

%There is a formal analogy with ordinary Fourier transform. The Fourier transform has a discrete torus version and a continuum version. As the scale of the torus becomes very big, the discrete version should pass to the continuum limit. In a similar way the Nahm transform should pass to the ADHM construction in the limit. Another feature of Fourier transform is that the double Fourier transform is local in nature. This gives us some hope that the Nahm tranform probably also has a local version.

\begin{abstract}
This paper develops a local analogue of the ADHM construction, which characterises ASD instantons defined over smooth bounded domains inside Euclidean $\R^4$ diffeomorphic to the 4-ball, in terms of infinite dimensional Hilbert spaces and bounded Hermitian linear operators satisfying an analogue of the ADHM equation. Morever, we describe the degeneration of this construction when a family of instantons develops a curvature singularity at the origin.
\end{abstract}

\section{Introduction}

This paper studies two related problems for ASD connections $A$ defined on a bounded region $B\subset \R^4$ with the standard Euclidean metric: \textbf{local Nahm transform}, and \textbf{singularity formation}. The unifying theme is to understand ASD connections by studying solutions to the coupled Dirac equation (\ie `Dirac fields').

The \textbf{Nahm transform} in various contexts is extensively studied (\cf e.g. \cite{DonaldsonKronheimer}\cite{BraamandBaal}\cite{Bartocci1}\cite{Cherkis}). It resembles the Fourier transform in many ways, and is also closely related to the celebrated ADHM construction. The prototype Nahm transform \cite{BraamandBaal}  constructs ASD connections over a given 4-torus, starting from coupled Dirac fields attached to ASD connections over the dual torus. It enjoys remarkable properties, such as
\begin{itemize}
	\item One-irreducible ASD connections are preserved under the transform.
	\item The inverse Nahm transform reproduces the original ASD connection.
\end{itemize}

It is desirable to establish a local version of Nahm transform; some motivations from the twistor space perspective are discussed in \cite{Witten}. Our main result in this direction, which borrows many techniques from \cite{BraamandBaal}, is

\begin{thm}\label{reconstructiontheorem}
(\textbf{Local Nahm transform}, \cf Chapter \ref{ThelocalNahmtransform} and \ref{Theinverseconstruction})
Let $B$ be a bounded region inside the Euclidean $\R^4$, such that $\bar{B}$ is diffeomorphic to the unit 4-ball with boundary. For a smooth ASD connection on a Hermitian vector bundle $(E,A)$ over $\bar{B}$, there is a  bundle of Hilbert spaces with an ASD connection $(\hat{E}, \hat{A})$ over the dual vector space $\hat{\R}^4$, called the local Nahm tranform. Morever, the inverse Nahm transform is canonically isomorphic to $(E, A)$.
\end{thm}

The local Nahm transform admits an alternative description analogous to the \textbf{ADHM construction}  \cite{DonaldsonKronheimer}. The information of the bundle $\hat{E}$ is equivalent to the Hilbert space of solutions to the coupled Dirac equations
\[
H^2_D= \{  s \in L^2(B, E\otimes S_-) : D_A^- s=0        \},
\]
which we call the `\textbf{Bergmann space}'. The information of the connection $\hat{A}$ is equivalent to 4 bounded Hermitian operators on $H^2_D$, defined by
\[
\hat{x}_\mu= P_0\circ x_\mu, \quad \mu=1,2,3,4,
\]
where $x_\mu$ means multiplication by the coordinate function, and $P_0$ is the orthogonal projection from $L^2(B, E\otimes S_-)$ to $H^2_D$. The ASD condition on $\hat{A}$ is equivalent to 
\[
\begin{cases}
[\hat{x}_1, \hat{x}_2]+[\hat{x}_3, \hat{x}_4]=0, \\
[\hat{x}_1, \hat{x}_3]+[\hat{x}_4, \hat{x}_2]=0, \\
[\hat{x}_1, \hat{x}_4]+[\hat{x}_2, \hat{x}_3]=0,
\end{cases}
\]
analgous to the \textbf{ADHM equation}. Thus $(\hat{E}, \hat{A})$ is analogous to the ADHM data, albeit $H^2_D$ is infinite dimensional, which brings forth some nuanced functional analytic features.

The inverse Nahm transform can also be equivalently described in this ADHM language. Given the ADHM data as above, the inverse Nahm transform bundle $\hat{\hat{E}}$ is the subbundle of the trivial bundle $H^2_D\otimes S_-$, defined as
\[
\hat{\hat{E}}_y=\ker [   H^2_D\otimes S_- \xrightarrow{-2\pi i \sum(\hat{x}_\mu-y_\mu)\hat{c}_\mu } H^2_D\otimes S_+            ], \quad \forall y\in B,
\]
and we equip $\hat{\hat{E}}\subset H^2_D\otimes S_-$ with the subbundle connection $\hat{\hat{A}}$, which turns out to be again ASD. The main content of Theorem \ref{reconstructiontheorem} is that $(\hat{\hat{E}}, \hat{\hat{A}})$ reconstructs $(E,A)$. Thus we have related ASD connections on domains to operator theory. M. Atiyah suggests to the author that this picture may be related to Alain Connes's noncommutative geometry \cite{AlainConnes}.

The second problem studied in this paper is \textbf{singularity formation}. 
 
\begin{Problem}
Given a 1-parameter family $\{A_t\}_{t>0}$ of smooth ASD connections $A_t$ on the Hermitian bundle $E$ over the closed ball $\bar{B}=\overline{B(R)}$, such that as $t\to \infty$, the connections converge smoothly away from the origin to $A_\infty$, and near the origin they are allowed to develop a curvature singularity. Describe the limiting behaviour of the Bergmann spaces $H^2_{D_{A_t}}$ of Dirac fields for $A_t$ as $t\to \infty$.
\end{Problem}

\begin{rmk}
The singularity formation problem is local in nature, so we assume $B$ is a ball for convenience. Morever, by possibly shrinking the ball, it makes sense to assume $A_\infty$ has small $L^2$ curvature.
\end{rmk}

We develop a convergence theory in Chapter \ref{ThespectralproblemChapter} and \ref{ConvergencetheoryforBergmannspaces} (which can be read independent of Chapter \ref{ThelocalNahmtransform} and \ref{Theinverseconstruction}). The Hamiltonian of the harmonic oscillator  \[
H(s)=\frac{1}{2R^2}\int_{B(R)}|x|^2|s|^2\] on $H^2_{D_{A_t}}$ induces a Hermitian operator with discrete and non-negative spectrum. The \textbf{spectral theory} is our key tool to understand the spaces $H^2_{D_{A_t}}$. There is a positive number called the spectral gap, such that for sufficiently concentrated ASD connections, the eigenvalues are either greater than the gap, which form the so called large spectrum, or very close to zero, which form the so called small spectrum, for which the eigenstates are localised near the origin.

The dimension of the small spectrum is constant for large $t$. To describe the limiting behaviour of the large spectrum, we recall that Uhlenbeck's removable singularity theorem allows us to smoothly extend the ASD connection $A_\infty$ across the origin to a smooth connection on a different bundle $\tilde{E}$ over  $B$. Then the large spectrum for $A_t$ converges in a natural sense to the spectrum of $A_\infty$. Morever,
 
 \begin{thm}
 	(\textbf{Convergence theory})
 	There is a natural topological bundle over $0<t\leq \infty$, whose fibres over $0<t<\infty$ are $H^2_{D_{A_t}}$ and the fibre over $\infty$ is the orthogonal direct sum $H^2_{D_{A_\infty}}\oplus V$, where $V$ has the dimension of the small spectrum. Morever, certain natural operators on the spaces $H^2_{D_{A_t}}$ can be extended continuously to the limit $H^2_{D_{A_\infty}}\oplus V$. 
 \end{thm}
 
\begin{rmk}
From the Nahm transform perspective, the fact that the bundle of Hilbert spaces splits into a direct sum, is analogous to the familiar picture of reducible connections.
\end{rmk}

An important quantity describing the  curvature singularity formation is the \textbf{instanton number}, \ie the amount of energy loss
\begin{equation}
k= \lim_{t\to \infty}\frac{1}{8\pi^2} \int_B |F_{A_t}|^2 -\frac{1}{8\pi^2}\int_B |F_{A_\infty}|^2.
\end{equation}
This is well known to have a more topological interpretation: for any $0<r\leq R$,
\begin{equation}\label{instantonnumberChernclass}
k= \frac{1}{8\pi^2}\int_{B(r)} \text{Tr}(F_{A_t}\wedge F_{A_t})-\int_{\partial B(r)} CS(A_t)
\end{equation}
where $CS(A)=\frac{1}{8\pi^2} \text{Tr}(dA\wedge A+\frac{2}{3}A\wedge A \wedge A)$ is the Chern Simons form, defined up to an integer, and becomes uniquely defined by demanding continuity at $t=\infty$:
\[
\lim_{t\to\infty} \int_{\partial B(r)} CS(A_t)= \int_{\partial B(r)} CS(A_\infty)= \frac{1}{8\pi^2} \int_{B(r)} \text{Tr}(F_{A_\infty}\wedge F_{A_\infty}).
\]

We combine our theory of the local Nahm transform and our analytic theory of singularity formation, together with some rudiments of K-theory, to show

\begin{thm}\label{instantonnumberequalsdimsmallspectrum} (\cf Section \ref{Singularityformationandsmallspectrum})
The instanton number is equal to the dimension of the small spectrum $\dim V$.
\end{thm}

\begin{rmk}
This is reminicent of descriptions of degenerate cases of the Nahm transform, where curvature singularity is switched with reducible connections.
\end{rmk}

\begin{Acknowledgement}
This work is part of the author's PhD first year minor projects, supervised by S.K. Donaldson and J.D. Lotay. The author also thanks S. Salamon for discussions, and
the London School of Geometry and Number Theory (Imperial College London, UCL and KCL) for providing a stimulating research environment.
 
This work was supported by the Engineering and Physical Sciences Research Council [EP/L015234/1], the EPSRC Centre for Doctoral Training in Geometry and Number Theory (The London School of Geometry and Number Theory), University College London. The author is also funded by Imperial College London for his PhD studies.
\end{Acknowledgement}

\section{The local Nahm transform}\label{ThelocalNahmtransform}

We define the local Nahm transform, and explain two different perspectives, which exhibit its analogy with the Nahm transform on the 4-torus and the ADHM construction of instantons on $\R^4$.

\subsection{Nahm transform perspective}

Suppose we are given a Hermitian bundle $E$ with a smooth ASD connection $A$ over the closure of a bounded connected domain $B\subset \R^4$ with smooth boundary, then we can produce an ASD connection on some (infinite rank) Hilbert bundle $\hat{E}\to \hat{\R^4}$ over the dual vector space $\hat{\R^4}$ as follows. First, we recall the \textbf{Poincar\'e bundle} $\mathcal{P}$ over $\R^4\times \hat{\R^4}$, given by equipping the trivial line bundle with the connection
\begin{equation}
\omega=2\pi i \sum_{\mu=1}^4 z_\mu dx_\mu,
\end{equation} 
where $x_\mu$ are the standard coordinates on $\R^4$ and  $z_\mu$ are the dual coordinates. For each fixed $z\in \hat{\R^4}$, the bundle $E$ couples to the flat line bundle $\mathcal{P}|_z$. Equivalently, one thinks of the coupled connection as a modified connection on $E$
\begin{equation}
A_z=A+2\pi i  \sum_{\mu=1}^4 z_\mu dx_\mu.
\end{equation} 
These $A_z$ are all gauge equivalent, but we keep them separate to make  apparent the analogy with the Nahm transform on the 4-torus \cite{BraamandBaal}. We introduce the \textbf{coupled Dirac operators}:
\begin{equation}
\begin{cases}
D^+_{A_z}: W_0^{2,1} ( B, E\otimes S_+) \rightarrow L^2 ( B, E\otimes S_-),
\\
D^-_{A_z}: L^2 ( B, E\otimes S_-) \rightarrow W^{2,-1} ( B, E\otimes S_+).
\end{cases}
\end{equation}
The convention in this paper is that the Clifford operators are anti-self-adjoint, with Clifford relations
\[
c_\mu c_\nu+c_\nu c_\mu= -2\delta_{\mu\nu},
\]
so the Dirac operator is formally self adjoint.

The Weitzenb\"ock formula implies that  $\ker D_{A_z}^+$ is zero. We define a bundle $\hat{E}$ over $\R^4$ with fibres given by all the $L^2$ solutions to the Dirac equation 
\begin{equation}
\hat{E}_z=\ker  D_{A_z}^-.
\end{equation}
This sits inside the trivial Hilbert bundle over $\hat{\R^4}$ with fibre $L^2( B, E\otimes S_-  )$. Let $P_z$ be the projections to the kernels according to the orthogonal splitting
\begin{equation}
L^2( B, E\otimes S_-  )= \hat{E}_z\oplus D^+_{A_z} W_0^{2,1}.
\end{equation}
The Sobolev space $W^{2,1}_0$ here imposes zero boundary condition; the subspace $D^+_{A_z}W^{2,1}_0$ is closed. The projection operator $P_z$ admits the formula
\begin{equation}
P_z=1-D^+_{A_z} G_z D^-_{A_z},
\end{equation}
where $G_z$ is the \textbf{Green's operator} solving the Dirichlet problem for the coupled Laplacian
\[
G_z=(\nabla_{A_z}^*\nabla_{A_z})^{-1} : W^{2,-1}(B, E)\to W^{2,1}_0(B,E).  
\]
Then $\hat{E}$ inherits the subbundle connection $\hat{A}$, with covariant derivative given by $\hat{\nabla}=P_z\hat{d}$, where $\hat{d}$ is the trivial connection.

\begin{defn}
The \textbf{local Nahm transform} of the pair $(E,A)$ is the Hilbert bundle with a connection $(\hat{E}, \hat{A})$ over $\hat{\R^4}$.
\end{defn}

\begin{prop}
The local Nahm transform remains ASD.
\end{prop}
	
\begin{proof}
(Sketch) The same proof in \cite{BraamandBaal}, Theorem 1.5 works (but their curvature formula is off by a sign). This goes by taking a local orthonormal framing of $\hat{E}$, denoted $\hat{f}^j(z)=\psi_z^j(x)$. Then one computes the connection coefficients 
\begin{equation}
\hat{A}_{jk}=\langle \hat{f}^j, \hat{d}\hat{f}^k \rangle,
\end{equation}
where the inner product comes from $L^2$ integration. 
The curvature components are
\begin{equation}\label{curvatureofNahmtransform}
\hat{F}_{ij}=\hat{d}{\hat{A}_{ij} }+\sum_k \hat{A}_{ik}  \wedge \hat{A}_{kj} =
-\langle \Omega \wedge \Omega\cdot \psi^i_z, G_z \psi^j_z \rangle,
\end{equation}
where
\[
-\Omega \wedge \Omega\cdot \psi^i_z=(2\pi)^2 \sum_{\mu,\nu} dz_\mu \wedge dz_\nu c_\mu c_\nu \psi^i_z.
\]
The curvature formula shows the connection $\hat{A}$ to be ASD.
\end{proof}

\subsection{ADHM perspective}

The local Nahm transform can also be cast in a form which makes apparent the analogy with the ADHM data \cite{DonaldsonKronheimer}.

By gauge equivalence, it is easy to see $\hat{E}_z=\exp(-2\pi i z) \hat{E}_0$. This means we can alternatively think about the underlying bundle of $\hat{E}$ as the trivial bundle with fibre isomorphic to $H^2_D=\hat{E}_0=\ker D^-_A$; this assigns a \textbf{canonical trivialisation} to $\hat{E}$. Notice that for $s\in H^2_D$,
\[
\begin{split}
	P_z \hat{d}(\exp(-2\pi i z) s   )=-2\pi i P_z(x_\mu dz_\mu \exp(-2\pi i z) s )
	\\
	=
	-2\pi i \exp(-2\pi i z) P_0(x_\mu  s )dz_\mu, 
	\end{split}
	\]
	so after identifying all fibres with $H^2_D$, the information of the connection $\hat{A}$ is up to some constant given by the operators $\hat{x}_\mu=P_0\circ x_\mu$, defined in essentially the same way as the matrices appearing in the \textbf{ADHM data}, except that these operators act on infinite dimensional spaces. 
	
\begin{lem}
(Analogue of the ADHM equation)
The ASD condition on $\hat{A}$ is equivalent to 
\begin{equation}\label{ADHMequation}
\begin{cases}
[ \hat{x}_1, \hat{x}_2 ]+[\hat{x}_3, \hat{x}_4]=0
\\
[ \hat{x}_1, \hat{x}_3 ]+[\hat{x}_4, \hat{x}_2]=0 \\
[ \hat{x}_1, \hat{x}_4 ]+[\hat{x}_2, \hat{x}_3]=0.
\end{cases}
\end{equation}
\end{lem}

\begin{proof}
By the above discussions, the connection matrix of $\hat{A}$ in the canonical trivialisation is the constant matrix valued 1-form $\hat{A}=-2\pi i\sum_\mu \hat{x}_\mu dz_\mu$ on $\hat{\R^4}$. Thus the curvature
$\hat{F}=\hat{A}\wedge \hat{A}$, and the result is clear after expanding this.
\end{proof}

\begin{rmk}
A useful intuition when one tackles analytical questions concerning the space $H^2_D$ of solutions to the coupled Dirac equation, is the analogy between the Dirac equation and the Cauchy-Riemann equation. The space of $L^2$ solutions to the Cauchy-Riemann equation is sometimes called the \textbf{Bergmann space}, which will be our terminology for $H^2_D$.
\end{rmk}

\begin{rmk}
In the next Chapter we will be concerned with reconstructing the original ASD connection from these ADHM data. We give some physical intuitions here about why such reconstructions are possible. The connection $A$ can be thought of as some gauge field, and physically this is detected by observing the motion of fermions inside the gauge field. The space $H^2_D$, \ie solutions to the coupled Dirac equation, can be viewed as the Hilbert space for the fermion, and operators such as $P_0\circ x_\mu$ are the natural ingredients to specify the quantum mechanics for the fermion.
\end{rmk}

\section{The inverse construction}\label{Theinverseconstruction}

We show how to reconstruct the original ASD connection from its local Nahm transform. We proceed by defining the inverse Nahm transform, interpret it in the frameworks of both the Nahm transform and the ADHM construction, and then define a canonical comparison map between the original bundle and the inverse Nahm transform bundle, which turns out to preserve both the Hermitian structure and the connection matrix.

\subsection{Inverse Nahm transform}

We wish to reconstruct the original bundle by an analogue of the inverse Nahm transform in \cite{BraamandBaal}. This suggests us to look at solutions to the Dirac equation coupled to the infinite rank bundle $\hat{E}$. To set up, we take the dual Poincar\'e bundle $\hat{\mathcal{P} }$ over $\hat{\R^4}\times \R^4$, with connection
\[
\hat{\omega}=2\pi i \sum_{\mu=1}^4  x_\mu dz_\mu,
\]
and for every $x\in \R^4$, we couple $\hat{E}$ to $\hat{\mathcal{P} }|_x$. Alternatively, one thinks of the coupled connection as $\hat{A}_x=\hat{A}+2\pi i \sum_{\mu=1}^4 x_\mu dz_\mu$.  We identify the spinors on $\R^4$ with those on $\hat{\R^4}$.

The sections of $\hat{E}$ which are constant under the canonical trivialisation $\hat{E}\simeq H^2_D$ are called \textbf{invariant}. One can think of this condition as the analogue of the periodicity condition over the dual 4-torus, when the size of the dual torus has shrunk to zero. It is clear that the covariant derivative on $\hat{E}$ preserves the invariance condition.

Then we define the fibres of the \textbf{inverse Nahm transform bundle} $\hat{\hat{E} }$ over $x\in\R^4$ to be the space of invariant solutions to the coupled Dirac operators $\hat{D}^-_{\hat{A}_x}$.  A priori the inverse Nahm transform bundle may have singularities. An invariant section $s$ in $\Gamma (\hat{\R^4}, \hat{E}\otimes S_-)$ has constant modulus for all fibres $z\in \hat{\R^4}$, so the Hermitian metric on any fibre can be taken to give a norm on $s$. This equips $\hat{\hat{E} }$ with a Hermitian structure. There is also a natural connection on $\hat{\hat{E} }$, given by the projection of the natural differentiation in the $x$ variable, where we view a section of $\hat{\hat{E} }$ as a function from $\R^4$ to the invariant part of $\Gamma (\hat{\R^4}, \hat{E}\otimes S_-)$.

\subsection{A formal comparison with ADHM construction}\label{formalcomparisonwithADHM}

The Nahm transform formulation  emphasizes the analogy with the torus case, and retains as much symmetry between the Nahm transform with its inverse as is possible. We can also view the inverse Nahm transform from the perspective of the ADHM construction. The dictionary for this interpretation, is that the invariant quantities on $\hat{\R^4}$ are replaced by the information on the zero fibre. The correspondence is:
\begin{itemize}
	\item The space of invariant sections of $\hat{E}$ corresponds to $H^2_D$.
	\item The connection $\hat{A}$ correspond to four operators on $H^2_D$: the covariant derivatives $\hat{\nabla}^\mu$ become $-2\pi i \hat{x}_\mu=-2\pi i P_0\circ x_\mu$, the projection of the multiplication by $x_\mu$.
	\item The fibre of the inverse Nahm transform $\hat{\hat{E}}_y$ becomes \[
	\ker [ H^2_D\otimes S_- \xrightarrow{-2\pi i \sum(\hat{x}_\mu-y_\mu)\hat{c}_\mu } H^2_D\otimes S_+] .
	\]
	\item The Hermitian metric on the inverse Nahm transform bundle is always inherited from the metric on $H^2_D$.
	\item The connection on the inverse Nahm transform is the subbundle connection for $\hat{\hat{E}} \subset H^2_D\otimes S_-$, where $H^2_D\otimes S_-$ is thought of as the trivial bundle over $B$.
\end{itemize}

%This is still different from the ADHM construction in two essential ways. The first is that the space $H^2_D$ is infinite dimensional. The second is that we haven't seen how the asymptotic data in the ADHM construction appears. What we need to achieve, is to provide a finite dimensional reduction mechanism, which morally says that if the connection $A$ has sufficiently concentrated curvature in a very small region near the origin, then the infinite amount of information in $H^2_D$ effectively collapses to a finite amount.

\begin{prop}(Analogue of ADHM construction)
Given the Hilbert space $H^2_D$ and four bounded Hermitian operators $\hat{x}_\mu$ on $H^2_D$ satisfying the ADHM type equation (\ref{ADHMequation}). Assume 
\[
\hat{\Lap}_{\hat{A}_y}=\hat{\nabla}_{\hat{A}_y}^*  \hat{\nabla}_{\hat{A}_y}= 4\pi^2 \sum_{\mu} (\hat{x}_\mu-y_\mu)^2, \quad \forall y\in B
\]
is invertible, then $\hat{\hat{E}}$ is a nonsingular Hermitian bundle over $B$, and
the subbundle connection $\hat{\hat{A}}$ on $\hat{\hat{E}}\subset H^2_D\otimes S_-$ is ASD.
\end{prop}

For a sketch proof of this fact in the Nahm transform language, see Corollary \ref{doubletransformisASD} below. The Proposition does not depend on knowing that the ADHM data arise from $(E,A)$, although the invertibility of $\hat{\Lap}_{\hat{A}_y}$ turns out to be automatic in that situation.

\subsection{Construction of Dirac fields}

We proceed as in \cite{BraamandBaal}. We aim to produce coupled Dirac fields for $\hat{A}$ using the Green's operator $G_z$ in order to define a canonical comparison map $E\to \hat{\hat{E}}$.

Notice there is a tautological element $\Psi \in \Gamma( B\times \hat{\R^4}, \hat{E}^*\otimes E\otimes S_- )$, given by evaluating any element of $\hat{E}$ on $E\otimes S_- $. If we further compose with the Green's operator $G_z$, we obtain a section
\[
G\Psi=\sum_{j=1}^\infty (G_z \psi_z^j)\hat{f^j}^* \in \Gamma( B\times \hat{\R^4}, \hat{E}^*\otimes E\otimes S_- ),
\]
where $\hat{f^j}^*$ stands for the dual basis of $\hat{f}^j=\psi_z^j$. Now suppose we have $f\in E_x^*$, then evaluating against $f$ produces an element in $ \Gamma( \{x \}\times \hat{\R^4}, \hat{E^*}\otimes S_-  )   $; denote it as $G\Psi(f)$.
We notice that for each $z\in \hat{\R^4}$, the value of $G\Psi(f)$ at the point $z$ is finite. This is because if we contract by a unit spinor $\eta$,
\[
|\langle G\Psi(f)(z),\eta \rangle|^2=\sum_{j} |\langle (G_z\psi_z^j)(f),\eta \rangle|^2
\]
is in fact the norm square of the linear functional $\langle f\circ G_z ,\eta \rangle$ on the Hilbert space $\hat{E}_z$. Since the connection $A$ is smooth, this linear functional is bounded by basic elliptic analysis.

\begin{prop}
The element $G\Psi(f)$ lies in $\ker \hat{D}^-_{\hat{A}^* }$, the kernel of the Dirac operator on the dual bundle of $\hat{E}$. 
\end{prop}

\begin{proof}
(Sketch) The formal calculation for the Dirac equation is the same as in \cite{BraamandBaal}, Proposition 2.1, which uses explicit expressions for the connection coefficient of $\hat{A}$, and reduces the proof to properties of the Green's operator.
\end{proof}

\begin{cor}
The element $4\pi \exp(2\pi i z(x) ) G\Psi(f)(z)$  lies in $\ker \hat{D}^-_{\hat{A}_x^* }$.
Morever, under the canonial trivialisation
$\hat{E}_z^*=\exp(2\pi i z) (H^2_D)^* ,$ this is invariant.

\end{cor}

\begin{proof}
Using $\hat{D}^-_{\hat{A}_x^* }=\hat{D}^-_{\hat{A}^* }-2\pi i\sum_{\mu=1}^4 x_\mu \hat{c}_\mu$, the first claim is clear.  
We consider the $z$ dependence of the construction. 
\[
\begin{split}
\exp(2\pi i z(x) ) G\Psi(f)(z)&=\sum f\circ (G_z\psi^j_z)\hat{f^j}^*(z) \exp(2\pi i z(x))
\\
&=\sum f\circ (\exp(-2\pi iz)G_0\psi^j_0)\hat{f^j}^*(z) \exp(2\pi i z(x)) \\
&=\exp(2\pi i z)\sum f\circ (G_0\psi^j_0)\hat{f^j}^*(0) ,
\end{split}
\]
where we used the fact the $f\in E_x^*$ is located above the point $x\in B$.
This means the element is invariant.
\end{proof}
 
Recall there is a complex antilinear automorphism of the spin bundle, denoted $\epsilon: S\rightarrow S$, which preserves the chiral splitting $S=S_+\oplus S_-$ and the Clifford multiplication. Tensoring this with the canonical isomorphism $E\simeq E^*$, one obtains an antilinear isomorphism
\begin{equation}
\hat{\epsilon}: E\otimes S\simeq E^*\otimes S,
\end{equation}
which is checked to preserve the Dirac equation.

This allows us to define a \textbf{canonical comparison map} $\alpha: E \rightarrow \hat{\hat{E} }$ by
\begin{equation}
\xi \mapsto  4\pi \exp(-2\pi i \langle \cdot , x \rangle)   \langle \hat{\epsilon}(G\Psi), \xi \rangle=4\pi \exp(-2\pi i \langle z, x \rangle)   \sum_j \langle \hat{\epsilon}(G_z \psi^j_z), \xi \rangle f^j(z),
\end{equation}
\ie we follow the composition $E_x \simeq E^*_x \rightarrow \ker \hat{D}^-_{\hat{A}_x^*} \simeq \ker \hat{D}^-_{\hat{A}_x}$.

The adjoint operator of this canonical map $E_y\to \hat{\hat{E}}_y\subset H^2_D\otimes S_-$ can be described as follows. Identify the space of invariant sections of $\hat{E}$ with $H^2_D$, then for $\zeta \in H^2_D\otimes S_-$, the adjoint operator maps it into an element of $E_y$, given by
\[
\zeta \mapsto 4\pi (\text{Tr}_{S_-}(\epsilon \circ G_0) \zeta )(y), \quad H^2_D\otimes S_-\to E_y,
\]
where the $\epsilon\circ G_0$ sends the $H^2_D$ factor to $\Gamma(E)\otimes S_-$, so we can then contract the two copies of $S_-$ to get some element in $\Gamma(E)$, and finally evaluate at the point $y\in B$.

\subsection{An example: the trivial line bundle}

Let us work out explicitly what this construction means for $(E,A)$ being the trivial flat line bundle, and $B=B(R)\subset \R^4$ is a ball. Then we can pick $\xi$ to be a unit length basis vector, and henceforth suppress it. Modulo the twisting $\epsilon$, the canonical map $E\rightarrow H^2_D\otimes S_-$ essentially comes down to the following description: calculate the composition  
\[
H^2_D \xrightarrow{G_0} W^{2,1}_0\cap C^\infty(B, E\otimes S_-) \xrightarrow{\text{evaluate} } E_y\otimes S_-\simeq S_-,
\]
and then write this (spinor-valued) functional on $H^2_D$ using Riesz representation.

The Green's function for the Euclidean Laplacian on the ball is well known. We recall the formula
\[
\begin{cases}
G(p,0)=\frac{1}{4\pi^2} \{  |p|^{-2}-R^{-2}     \} \\
G(p,q)=\frac{1}{4\pi^2} \{  |p-q|^{-2}-|p\frac{|q|}{R}-q\frac{R}{|q|}            |^{-2}     \}=G(q,p).
\end{cases}
\]
We treat $p$ as the independent variable, and $q$ is a parameter.
In the spinorial situation, the Green's function is merely the above tensored by a parallel spinor field $\eta$. We calculate the projection of the Green's function to the space $H^2_D$,
\[
\begin{cases}
G(p,0)\otimes \eta=D \{  -\frac{p\cdot \eta}{8\pi^2} (  |p|^{-2}-R^{-2}                )   \}+ \frac{\eta}{4\pi^2 R^2} \\
G(p,q)\otimes \eta =D \{  -\frac{1}{2} (p-q)\cdot \eta G(p,q)                                            \}+ \frac{1}{4\pi^2} (\frac{R^2}{|q|^2}-1) \frac{R^2}{|q|^2} \frac{p-q\frac{R^2}{|q|^2} }{|  p-q\frac{R^2}{|q|^2}         |^{4}  }\cdot    q\cdot \eta.
\end{cases}
\]

This means for $s\in H^2_D$,
\[
\begin{split}
\hat{\epsilon}Gs(q)=& 
\hat{\epsilon} \int_{B(R)}    G(p,q )s(p)dp \\
=& \hat{\epsilon} \sum_{\nu=1}^{2}\int_{B(R)}  \eta_\nu \otimes \langle G(p,q )\eta_\nu , s(p) \rangle dp \\
=&   \sum_{\nu=1}^{2}  \hat{\epsilon}\eta_\nu \otimes \langle  
\frac{1}{4\pi^2} (\frac{R^2}{|q|^2}-1) \frac{R^2}{|q|^2} \frac{p-q\frac{R^2}{|q|^2} }{|  p-q\frac{R^2}{|q|^2}         |^{4}  } \cdot   q\cdot \eta_\nu
, s(p) \rangle .
\end{split}
\]
Thus the image of the canonical map $\alpha: E_y \to \hat{\hat{E}}_y\subset H^2_D\otimes S_-$ is
\begin{equation}
\alpha(1)=\begin{cases}
\sum_{\nu} \epsilon{\eta_\nu} \otimes \frac{1}{\pi} (\frac{R^2}{|y|^2}-1) \frac{R^2}{|y|^2} \frac{x-y\frac{R^2}{|y|^2} }{|  x-y\frac{R^2}{|y|^2}         |^{4}  } \cdot   y\cdot \eta_\nu \in S_-\otimes H^2_D \quad & 0<|y|<R,\\
\sum_{\nu} \epsilon{\eta_\nu} \otimes \frac{\eta_\nu}{4\pi R^2}\quad&  y=0.
\end{cases}
\end{equation}
Geometrically, the Dirac field $\frac{x-y\frac{R^2}{|y|^2} }{|  x-y\frac{R^2}{|y|^2}         |^{4}  } \cdot   y\cdot \eta_\nu$ is proportional to the fundamental solution to the Dirac equation with source at the inversion point of $y$; when $y\to 0$ the source moves to infinity and the suitably rescaled solution converges smoothly to the constant Dirac field.

It is interesting to observe the limiting behaviour of this formula near the boundary of $B(R)$. For $|y|\sim R$, write $y=Rl\cos\theta$, where $\theta\sim 0$, and $l$ is a unit vector. Write also $x=lR+\frac{1}{2} x'R\theta^2$, then the formula has the asymptote
\begin{equation}
\sum_{\nu} \epsilon{\eta_\nu}  \otimes \frac{1}{\pi} \theta^2 \frac{x-lR-\frac{1}{2} lR\theta^2  } {| x-lR-\frac{1}{2} lR\theta^2         |^{4}  } \cdot   lR\cdot \eta_\nu
\sim  
\sum_{\nu} \epsilon{\eta_\nu}  \otimes \frac{8}{\pi R^2 \theta^4}  \frac{x'-l } {|  x'-l      |^{4}  } \cdot   l\cdot \eta_\nu.      
\end{equation}
If we view this na\"ively as a function of $x$, then the pointwise limit will be zero, because the Dirac field concentrates near $lR$ with distance scale $\theta^2$. if we extract an appropriately rescaled limit, we obtain a Dirac field in some half space, generated by a point charge placed at unit distance outside the boundary plane. 

\subsection{Nonsingularity}

%We next show that $\hat{\nabla}_{\hat{A}}^*
%\hat{\nabla}_{\hat{A}}$ admits an inverse, by explicit construction. First, we define the integral kernel
%\begin{equation}
%\hat{G}(z_1,z_2)_{ij}=\frac{1}{4\pi^2|z_1-z_2|^2}    \int_{B(R)} \langle \psi^i_{z_1}(x), \psi^j_{z_2}(x) \rangle dx.
%\end{equation}
%This gives the matrix entries of an element in $\hat{E}_{z_1}\otimes \hat{E}^*_{z_2} $. A more intrinsic description is to take any $f\in \hat{E}_{z_2}\subset L^2(B(R), E\otimes S_-)$, and then 
%\begin{equation}
%\hat{G}(z_1, z_2)f=\frac{1}{4\pi^2|z_1-z_2|^2} P_{z_1}f \in \hat{E}_{z_1}.
%\end{equation}

This Section studies the \textbf{Laplacian} $\hat{\Lap}_{\hat{A}_y}=\hat{\nabla}_{\hat{A}_y}^*\hat{\nabla}_{\hat{A}_y}$ on the local Nahm transform bundle $\hat{E}$, where $y\in \R^4$. The importance of this operator for our construction is based on the observation that since $\hat{A}_y$ is ASD, on the positive spin $\hat{E}\otimes S_+$, we have
\begin{equation}
\hat{\nabla}_{\hat{A}_y}^*\hat{\nabla}_{\hat{A}_y}=\hat{D}_{\hat{A}_y}^-\hat{D}_{\hat{A}_y}^+.
\end{equation}

We start with a matrix description of this Laplacian acting on invariant sections of $\hat{E}$. Recall that $s\in H^2_D$ is canonically identified with the invariant section $s e^{-2\pi i z}$.

\begin{lem}\label{LaplacianonNahmtransform}
Under the identification, the Laplacian can be identified with the operator $4\pi^2 P_0 \circ \{|x-y|^2 -4G_0\}$ acting on $H^2_D$.	
\end{lem}

\begin{proof}
Let $s\in H^2_D$ and $\tau=G_0s$. Then by Green's formula
\[
\int_{B}  \langle \tau, s \rangle=\int_{B}   |\nabla \tau|^2 .
\]
We consider the $L^2$ decomposition 
\begin{equation}\label{L2decompositionequation1}
(x_\mu-y_\mu) s= D \tau_\mu+\sigma_\mu,
\end{equation}
with $\tau_\mu \in W^{2,1}_0$ and $D \sigma_\mu=0$. Then we observe \[
\nabla^*\nabla \tau_\mu=D^2 \tau_\mu=D( (x_\mu-y_\mu) s   )=c_\mu s=\nabla^*\nabla c_\mu \tau,
\]
 so in fact, by the uniqueness of solutions to the Laplace equation, 
\begin{equation}\label{taumu}
\tau_\mu=c_\mu \tau.
\end{equation} Then we compute
\[
\norm{(x_\mu-y_\mu)s}_{L^2}^2=\norm{D \tau_{\mu}}_{L^2}^2 +\norm{\sigma_\mu}_{L^2}^2=\norm{\nabla \tau}_{L^2}^2 +\norm{\sigma_\mu}_{L^2}^2
\]
Summing over $\mu$, we get
\[
\int_{B} |x-y|^2 |s|^2=4\int_{B}   |\nabla \tau|^2 +\sum_{\mu} \norm{\sigma_\mu}_{L^2}^2= 4 \langle \tau, s \rangle+  \sum_{\mu} \norm{\sigma_\mu}_{L^2}^2.
\]
This equation has the interpretation
\begin{equation}\label{LaplacianonNahmtransform2}
	\langle s, P_0(   \{ |x-y|^2-4G_0\}  s      )\rangle=\sum_{\mu} \norm{\sigma_\mu}_{L^2}^2
	= \sum_{\mu } \norm{ P_0\circ (x_\mu-y_\mu) s }_{L^2}^2
	,
\end{equation}
Under the identification, the directional derivatives $\nabla^\mu_{\hat{A}_y }$ can be identified as 
$-2\pi i P_0 \circ (x_\mu-y_\mu)$ (\cf Section \ref{formalcomparisonwithADHM}). Thus
\[
\langle s, 4\pi^2 P_0(   \{ |x-y|^2-4G_0\}  s      )\rangle= \sum_{\mu } \norm{ \nabla^\mu_{\hat{A}_y} s}^2= \langle s, \hat{\Lap}_{ \hat{A}_y }s \rangle.
\]
Since this holds for all $s$, the self adjoint operators $4\pi^2 P_0\circ  \{ |x-y|^2-4G_0\}$ and $\hat{\Lap}_{\hat{A}_y}$ must be the same.
\end{proof}

\begin{rmk}
One can also derive the Laplacian formula by a calculation similar to Lemma 2.3, Lemma 2.4 in \cite{BraamandBaal}.
\end{rmk}

These preparations lead to

\begin{prop}
The kernel of $\hat{\Lap}_{\hat{A}_y }$ on the space of invariant sections of $\hat{E}$ vanishes for $y\in \R^4$.
\end{prop}
\begin{proof}

An invariant section is the same data as an element $s\in H^2_D$. Then the kernel of $\hat{\Lap}_{\hat{A}_y }$ translates into the condition
\begin{equation}\label{kernelofLaplacian}
P_0(   \{ \sum_{\mu} (x_\mu-y_\mu)^2-4G_0\}  s      )=0.
\end{equation}
Comparing with the proof of Lemma \ref{LaplacianonNahmtransform}, this implies (\cf (\ref{LaplacianonNahmtransform2}))
\[
\sum_{\mu}  \norm{\sigma_\mu}_{L^2}^2=0,
\]
hence $\sigma_\mu=0$ for all $\mu$.
Now combining (\ref{L2decompositionequation1}), (\ref{taumu}),
\begin{equation}\label{LaplacianNahmtransform1}
\begin{split}
D(  \{ |x-y|^2-4G_0\}  s )
&=2\sum_{\mu } (x_\mu-y_\mu )c_\mu s-4 D\tau
=2 \sum_{\mu} c_\mu D\tau_\mu -4D\tau \\
&=2\sum c_\mu D(c_\mu \tau)-4D\tau=0,
\end{split}
\end{equation}
so $\{ |x-y|^2-4G_0\}  s $ is both orthogonal to $H^2_D$ (by equation (\ref{kernelofLaplacian})) and inside $H^2_D$ (by equation (\ref{LaplacianNahmtransform1})), and therefore
\[
|x-y|^2  s=4\tau.
\]
In particular $s$ has zero boundary data. Substituting this into (\ref{L2decompositionequation1}), we get
\[
\begin{split}
(x_\mu-y_\mu)s&=D(c_\mu \frac{1}{4}|x-y|^2 s)=-c_\mu D(\frac{1}{4}|x-y|^2 s) -2\nabla_\mu (\frac{1}{4}|x-y|^2 s) \\
&=-c_\mu \frac{1}{2}(x-y) \cdot s -(x_\mu-y_\mu)s- \frac{1}{2}|x-y|^2 \nabla_\mu s,
\end{split}
\]
So for any $\mu$,
\[
|x-y|^2 \nabla_\mu s=-4(x_\mu-y_\mu)s-\sum_{\nu}(x_\nu-y_\nu)c_\mu c_\nu s.
\]
From this 
\[
\nabla_{A} \{|x-y|^2 (x-y)\cdot s\}=0.
\]
In other words,
\begin{equation}\label{asymptoteofs}
s=\frac{(x-y)\cdot \rho}{|x-y|^4}.
\end{equation}
where $\rho$ is covariant constant in $E\otimes S^+$.

This forces $s=0$, by the zero boundary condition.
\end{proof}

\begin{rmk}
This proof suggests the mechanism of singularity formation: if the discussion is extended to singular connections, then the kernel of $\hat{\Lap}_{\hat{A}_y }$ can be nonzero, and (\ref{asymptoteofs}) is expected to describe the asymptotic profile of the kernel elements near the singularity.
\end{rmk}

Our next aim is to justify the invertibility of $\hat{\Lap}_{\hat{A}_y }$, acting on the space of invariant sections of $\hat{E}$, identified as $H^2_D$. We recall the Laplacian is always self-adjoint and semipositive. In our context it is also bounded. So it is enough to prove

\begin{lem}
For $y\not \in \partial B$,
there is a coercive estimate on the space of invariant sections
\begin{equation}
\norm{\hat{\Lap}_{\hat{A}_y } (se^{-2\pi i z})} \geq C \norm{se^{-2\pi i z} }.
\end{equation}
Here the norm of an invariant section is the norm on any fibre $\hat{E}_z$. The constant can be taken to be uniform for $y$ in any compact subset of $B$.
\end{lem}
\begin{proof}
Since we already know the kernel is zero, the strategy for the coercive estimate is a compactness argument, which is delicate due to norm collapsing issues. Suppose we have a sequence $s_m \in H^2_D$, with unit norms, satisfying
\[
\norm{P_0( \{ |x-y|^2-4G_0 \} s_m )}_{L^2}=\frac{1}{4\pi^2}\norm{\hat{\Lap}_{\hat{A}_y } (s_me^{-2\pi i z})} \leq 1/m .
\]
By interior regularity, we may assume $s_m$ converges smoothly to $s$ inside $B$. We associate $\tau_m=G_0 s_m$ as in the previous proposition. Then $\tau_m$ satisfies the elliptic equation $\Lap_{A} \tau_m=s_m$, with the zero boundary condition, so for any large Lebesgue exponent $p$ of our choice,
 \[
\norm{\tau_m}_{L^p} \leq C \norm{\tau_m}_{W^{2,2}_0} \leq C \norm{s_m}_{L^2}\leq C.\]

For small $\delta$, let $B_\delta=\{   
x\in B: \text{dist}(x, \partial B  )<\delta
\}$. Let $q$ the conjugate exponent of $p$, then
\[
\begin{split}
&\int_{B} |x-y|^2 |s_m|^2(x) dx
=4\int_{B}  \langle \tau_m, s_m \rangle +\langle P_0(   \{ |x-y|^2-4G_0\}  s_m), s_m \rangle \\
&\leq C \norm{s_m}_{L^q(B)}+\frac{1}{4\pi^2}\norm{\hat{\Lap}_{\hat{A}_y }(s_me^{-2\pi i z})} \\
&\leq C \norm{s_m}_{L^q(B\setminus B_\delta)}+ C \norm{s_m}_{L^q(B_\delta) }+1/m \\
&\leq C \norm{s_m}_{L^q(B\setminus B_\delta)}+C\delta^{-1/2+1/q}\norm{s_m}_{L^2(B_\delta)      }+1/m
.
\end{split}
\]

In general, if a higher Lebesgue norm of a function can be controlled by a lower Lebesgue norm, then the support of the function can not have arbitrarily small volume, \ie the function cannot be concentrated in region of small measure. If we replace the Lebesgue norm by a weighted version, then the function cannot be concentrated where the weight is bounded positively below. 
Hence the norm of $s_m$ cannot be lost entirely to the boundary:
\[
\norm{s}_{L^q(B)}\geq \liminf \norm{s_m}_{L^q(B\setminus B_\delta)} \geq C\text{dist}(y, \partial B)^2 .
\]
where $C$ depends only on $R$ and the connection $A$.
In particular the limit $s$ is nontrivial. Now we look at the equation
\[
|x-y|^2 s_m-4\tau_m=P_0( |x-y|^2 s_m-4\tau_m )+D t_m,\quad \tau_m=G_0 s_m,
\]
where $t_m$ is just some element in $W^{2,1}_0$. When we take the weak limit, we get
\[
|x-y|^2 s-4\tau=D t,\quad \tau=G_0 s.
\]
Thus despite the possiblity of partial norm collapsing, the limit $s$ still satisfies 
\[
P_0( \{  |x-y|^2 -4G_0             \} s )=0,
\]
or in other words, \[
\hat{\Lap}_{\hat{A} _y} (se^{-2\pi i z})=0.
\]
which implies $s=0$, a contradiction.
\end{proof}

Hence we finally achieved
\begin{prop}
For $y\notin \partial B$, 
the Laplacian $\hat{\Lap}_{\hat{A}_y }$ is invertible on the space of invariant sections of $\hat{E}$.
\end{prop}

\begin{rmk}
The subtle argument in the lemma involving norm collapsing reflects the genuine distinction between interior points and the boundary. The coercivity estimate is expected to fail on the boundary, so that $\hat{\hat{E} }$ would not extend to a smooth bundle on ${\R^4}$; otherwise the original ASD connection $A$ would extend to a global connection on $\R^4$, which would be very surprising.
\end{rmk}

We can therefore define the inverse operator $\hat{G}_y=(\hat{\Lap}_{\hat{A}_y } )^{-1}$, for $y \notin \partial B$, on the space of invariant sections of $\hat{E}$. We collect a few formal consequences:

\begin{cor}\label{cokernelvanishingonNahmtransform}
The operator $\hat{D}^-_{\hat{A}_y }:   \Gamma( \hat{\R^4}, \hat{E}\otimes S_-   ) \rightarrow  \Gamma( \hat{\R^4}, \hat{E}\otimes S_+   )  $, when restricted to the invariant sections, is a surjective map. Morever, the orthogonal projection from $H^2_D\otimes S_-$ to the kernel $\hat{\hat{E}  }_y$ is given by
\begin{equation}\label{projectionoperatorformula1}
\hat{P}_y=1- \hat{D}^+_{\hat{A}_y }  \hat{G}_y            \hat{D}^-_{\hat{A}_y }.
\end{equation}
\end{cor}

\begin{proof}
Since $\hat{A}_y$ is ASD, we have $
\hat{\Lap}_{\hat{A}_y }=\hat{D}^-_{\hat{A}_y }\hat{D}^+_{\hat{A}_y }
$. The surjectivity follows from the explicit formula for a preimage, given by the operator $\hat{D}^+_{\hat{A}_y }\hat{G}_y $. To see the orthogonal projection formula, one also needs to check $\hat{D}^+_{\hat{A}_y }$ and $\hat{D}^-_{\hat{A}_y }$ are adjoint on the invariant sections.
\end{proof}

\begin{cor}\label{doubletransformisASD}
The bundle $\hat{\hat {E}}$ is nonsingular over $y\not \in \partial B$. Morever, the natural connection $\hat{\hat {A}}$ is ASD.
\end{cor}

\begin{proof}
The nonsingularity is a formal consequence of a smooth formula for the orthogonal projection $\hat{P}_y$; here the word bundle is interpreted as a Hermitian vector bundle with fibres possibly being Hilbert spaces.

The ASD condition is implied by the following curvature formula for $\hat{\hat{A}}$ analogous to (\ref{curvatureofNahmtransform}), which relies on formula (\ref{projectionoperatorformula1}) for the projection operator. 
At the point $y\in B$, we have  $\hat{\hat{F }} \in \End( \hat{\hat {E}}_y) \otimes \Lambda^2T_y^*(\R^4)     $,
\begin{equation}\label{curvatureofdoubleNahmtransform}
\hat{\hat{F }}= -(2\pi)^2 \hat{P}_y\hat{G}_y \otimes \sum_{\mu, \nu} \hat{c}_\mu\hat{ c}_\nu dx_\mu \wedge dx_\nu,
\end{equation}
where $\hat{P}_y\hat{G}_y$ acts on the $\Gamma(\hat{E})$ factor of $\hat{\hat {E}}_y$, and $\hat{c}_\mu\hat{ c}_\nu$ acts on the spin factor. 
\end{proof}

We next wish to show the inverse Nahm transform vanishes in the exterior of $B$. 

\begin{cor}\label{doubletransformvanishesintheexterior}
The inverse Nahm transform $\hat{\hat {E}}_y=0$ for $y\notin \bar{B}$.
\end{cor}

\begin{proof}
Since we know the bundle is nonsingular in the exterior region, the rank is constant, so it is only necessary to show $\hat{\hat {E}}_y=0$ for $|y|>>1$. But we know
\[
\hat{D}^-_{\hat{A}_y}=\hat{D}^-_{\hat{A}}+2\pi i \sum y_\mu \hat{ c}_\mu= 2\pi i  (1+ \hat{D}^-_{\hat{A}}  (2\pi i \sum y_\mu \hat{c} _\mu        )^{-1}           )\sum y_\mu \hat{ c}_\mu,
\]
and $\hat{D}^-_{\hat{A}}$ acting on invariant sections is a bounded operator,
so for large $|y|$, the Dirac operator is merely a small perturbation of the invertible operator $2\pi i\sum  y_\mu \hat{c}_\mu $, hence itself invertible, implying the kernel $\hat{\hat {E}}_y=0$.
\end{proof}

\subsection{The canonical map}

To study the canonical map $\alpha: E \rightarrow \hat{\hat{E }}$, we look at 
\begin{equation}
v_x(z)=4\pi \exp(-2\pi i z(x))\sum_j \hat{\epsilon}(G_z \psi^j_z(x))\otimes  \hat{f^j}(z) \in E_x^* \otimes S_-\otimes \hat{E}_z,
\end{equation}
where $x \in B$ is a parameter. For $x_1,x_2\in B$,
we consider the correlator function $\langle v_{x_1}(z), v_{x_2}(z) \rangle_{\hat{E}_z \otimes S_-}$, meaning we contract the $\hat{E}_z$ and the $S_-$ factor, to get a matrix in $\Hom(E_{x_2}, E_{x_1})$. We aim to derive a formula for the correlators, by adapting arguments in \cite{BraamandBaal}.

\begin{prop} (\cf \cite{BraamandBaal}, Lemma 2.6)
The correlator is
 \begin{equation}\label{Greenfunctionembeddingdata}
\langle v_{x_1}(z), v_{x_2}(z) \rangle_{\hat{E} \otimes S_-}= 4\pi^2 G_0(x_1,x_2)|x_1-x_2|^2.
\end{equation}
which is independent of $z$.
\end{prop}

\begin{proof}
Since $\langle \hat{\epsilon}(v), \hat{\epsilon}(w) \rangle=\langle w,v \rangle$, we calculate
\begin{equation}
\begin{split}
& \langle v_{x_1}(z), v_{x_2}(z) \rangle_{\hat{E}_z \otimes S_- }\\
=& (4\pi)^2 \exp(2\pi iz(x_1-x_2)) \langle \sum_j \hat{\epsilon} G_z\psi_z^j(x_1) \otimes f^j(z), \sum_k \hat{\epsilon} G_z\psi_z^k(x_2) \otimes f^k(z) \rangle_{\hat{E}_z \otimes S_-} \\
=& (4\pi)^2 \exp(2\pi iz(x_1-x_2))\text{Tr}_{S_-}\sum_j   |G_z\psi_z^j(x_1)  \rangle \langle  G_z\psi_z^j(x_2) |  .
\end{split}
\end{equation}
Here the bra-ket notation indicates $\langle  G_z\psi_z^j(x_2) |\in E_{x_2}^*$ and $|G_z\psi_z^j(x_1)  \rangle
\in E_{x_1}$.

The expression 
$
\sum_j   |G_z\psi_z^j(x_1)  \rangle \langle  G_z\psi_z^j(x_2) |
$
can be viewed as the Schwartz kernel of an operator acting on $\Gamma(B, E\otimes S_-)$, which converts a section $h$ with independent variable $x_2$ to a section with independent variable $x_1$. This can be factorised into several steps. The first step sends $h$ to 
\[
\sum_{k} \hat{f}^k(z) \int_{B} \langle  G_z \psi^k_z(x_2), h(x_2) \rangle dx_2
=P_zG_z (h).
\]
Here the equality uses the self-adjointness of $G_z$, and the decomposition formula for the projection  $P_z=\sum \hat{f^k}\langle \hat{f^k},\cdot \rangle$. The second step contracts the $\hat{E}_{z}$ factor with $\hat{f}^j$ to get a number, multiplies the result by $G_z\psi^j_z(x_1)$, and sums over $j$. Using again the decomposition formula for the projection, the result is 
\begin{equation}
\int_B \sum_j   |G_z\psi_z^j(x_1)  \rangle \langle  G_z\psi_z^j(x_2) | h(x_2) d\text{Vol}(x_2)=
ev_{x_1}  \circ G_z P_z  G_z  (h).
\end{equation}

Now we simplify the formula and take the trace over the spinor factor $S_-$. The projection can be written as $P_z=1-D^+_{A_z} G_z D^-_{A_z}$. Hence using that the Clifford multiplication $c_\mu c_\nu$  can contribute to the spinor trace if and only if $\mu=\nu$, we see that (using summation convention)
\[
\text{Tr}_{S_-} (G_z P_z G_z)=\text{Tr}_{S_-} (G_z^2-G_z c_\mu \nabla^\mu_{A_z} G_z c_\nu \nabla^\nu_{A_z} G_z)=\text{Tr}_{S_-} (G_z^2+G_z \nabla^\mu_{A_z} G_z \nabla^\mu_{A_z} G_z).
\]
Now we use the formula $G_z \nabla^\mu_{A_z} G_z=\frac{1}{4\pi i} \frac{\partial G_z}{\partial z_\mu}$ (see \cite{BraamandBaal}, Lemma 2.2) to simplify:
\[
\begin{split}
G_z \nabla^\mu_{A_z} G_z \nabla^\mu_{A_z} G_z&=\frac{1}{4\pi i} \frac{\partial G_z}{\partial z_\mu} \nabla^\mu_{A_z} G_z \\
&=\frac{1}{4\pi i}    \frac{\partial}{\partial z_\mu}  (G_z \nabla^\mu_{A_z} G_z)-\frac{1}{4\pi i} G_z \frac{\partial}{\partial z_\mu}  ( \nabla^\mu_{A_z} G_z)\\
&=\frac{-1}{(4\pi)^2}  \frac{\partial^2 G_z}{\partial z_\mu^2} -\frac{1}{4\pi i} G_z\nabla^\mu_{A_z} \frac{\partial G_z}{\partial z_\mu}  -\frac{1}{4\pi i} G_z [\frac{\partial}{\partial z_\mu},   \nabla^\mu_{A_z} ]G_z \\
&=\frac{-1}{(4\pi)^2}  \frac{\partial^2 G_z}{\partial z_\mu^2}-G_z \nabla^\mu_{A_z} G_z \nabla^\mu_{A_z} G_z-2G_z^2.
\end{split}
\]
From this we see 
\begin{equation}
\text{Tr}_{S_-} (G_z P_z G_z)=\text{Tr}_{S_-} (G_z^2+G_z \nabla^\mu_{A_z} G_z \nabla^\mu_{A_z} G_z)=\frac{-1}{(4\pi)^2}  \frac{\partial^2 G_z}{\partial z_\mu^2}.
\end{equation}
Here we need to take into account the fact that $S_-$ is 2-dimensional, so contributes twice to the trace. 

Combining the above, we get the formula for the correlator by comparing the Schwartz kernel of operators:
\[
\langle v_{x_1}(z), v_{x_2}(z) \rangle_{\hat{E}_z \otimes S_-}= - \exp(2\pi iz(x_1-x_2))\sum_{\mu} \frac{\partial^2 G_z}{\partial z_\mu^2}(x_1,x_2).
\]
Furthermore, we have $\Lap_{A_z}=e^{-2\pi i z}\Lap_{A}e^{2\pi i z},
$
so the Green's function is
$
G_z(x,y)=\exp(-2\pi i z(x-y))G_0(x,y).
$
From this we see
\begin{equation*}
\begin{split}
- \exp(2\pi iz(x-y))\sum_{\mu} \frac{\partial^2 G_z}{\partial z_\mu^2}(x,y)
=4\pi^2 G_0(x,y)|x-y|^2 
\end{split}
\end{equation*}
The claim follows.
\end{proof}

Given $\xi \in E_x$, the image under the canonical map $\alpha: E\rightarrow \hat{\hat{E }}$ is precisely the invariant section $\langle v_x, \xi \rangle \in \Gamma(\hat{\R^4}, \hat{E} \otimes S_-)$. We show that the canonical map $\alpha$ preserves the structures (\cf \cite{BraamandBaal}, Theorem 2.8 and 2.9).

\begin{prop}\label{comparisonmappreservesmetric}
The canonical map $\alpha$ preserves the Hermitian metric.	
\end{prop}

\begin{proof}
The Green's function has the short distance asymptotic expansion
\begin{equation}
G_z(x,y)=\frac{1}{4\pi^2|x-y|^2} \{ I-\sum_{\mu} (A_\mu(x)+2\pi i z_\mu)(x_\mu-y_\mu)   +O(|x-y|^2)     \}.
\end{equation}
From this we see
\begin{equation}\label{Greenfunctionembeddingdata1}
\begin{split}
\langle v_{x}(z), v_{y}(z) \rangle_{\hat{E} \otimes S_-}
=4\pi^2 G_0(x,y)|x-y|^2= I-\sum_{\mu} A_\mu(x)(x_\mu-y_\mu)   +O(|x-y|^2) .
\end{split}
\end{equation}

Thus the norm square of $\langle v_x, \xi \rangle$ can be calculated by first extending $\xi$ to be part of a local orthonormal frame, and then taking the limit $y\to x$. This yields $ |\langle v_x, \xi \rangle|^2=\langle \xi|I|\xi\rangle =|\xi|^2$, by looking at the lowest order in the expansion.
\end{proof}

\begin{prop}\label{comparisonmappreservesconnection}
The canonical map $\alpha$ preserves the 
connection matrix.
\end{prop}

\begin{proof}
Let $\xi$ and $\xi'$ be two local sections of $E$, fitting into a local orthonormal frame of $E$. These map to $v_x^\xi=\langle v_x, \xi \rangle$ and $v_x^{\xi'}=\langle v_x, \xi' \rangle$.
The connection matrix on $\hat{\hat{E }}$ at the point $x \in B$ is specified by knowing $\langle v_{x}^{\xi'}, \frac{d}{dy}|_x v_{y}^{\xi} \rangle (x)$.
But this can be calculated from the expansion of the Green's function above, which implies the required
\[
\langle v_{x}^{\xi'}, \frac{d}{dy}|_x v_{y}^{\xi} \rangle (x)=\sum_\mu \langle {\xi'}, A_\mu(x)  {\xi} \rangle dx_\mu=\langle {\xi'}, \nabla_{A}  {\xi} \rangle.
\]
\end{proof}

\subsection{The reconstruction theorem}

We derive the main reconstruction theorem \ref{reconstructiontheorem} assuming the following Lemma on Fredholm index, which we shall prove later in Section \ref{Therankofdoubletransformandindexcomputation}.

\begin{lem}\label{indexproblem}
Assume $\bar{B}$ is diffeomorphic to the unit 4-ball with boundary.
The index of the Dirac operator $\hat{D}^-_{\hat{A}_y}$ acting on the space of invariant sections equals $\text{rank}(E)$ for $y\in B$.
\end{lem}

\begin{thm}\label{Thereconstructiontheorem}(\textbf{Reconstruction})
Let $B$ be a bounded domain in $\R^4$, such that $\bar{B}$ is diffeomorphic to the unit 4-ball with boundary, and $A$ is a smooth ASD connection on $\bar{B}$.
The canonical map $\alpha: E\to \hat{\hat{E}}$ is an isomorphism over $B$, preserving the Hermitian metric and the connection, and in the exterior of $B$ the inverse Nahm transform $\hat{\hat{E}}$ vanishes.
\end{thm}

\begin{proof}
By Corollary \ref{doubletransformvanishesintheexterior}, the inverse Nahm transform vanishes in the exterior region. By Corollary \ref{cokernelvanishingonNahmtransform}, the operator $\hat{D}^-_{\hat{A}_y}$ is surjective, so its kernel dimension is equal to the index. For $y\in B$, by the Lemma above the index is $\text{rank }(E)$, so $\text{rank }(\hat{\hat{E}})= \text{rank }(E)$. By Proposition \ref{comparisonmappreservesmetric}, the canonical map $\alpha$ is an injective isometry, so must be an isomorphism. The connection $A$ agrees with $\hat{\hat{A}}$ by Proposition \ref{comparisonmappreservesconnection}.
\end{proof}

We now discuss the result from a number of perspectives.

\begin{rmk}
	The reconstruction theorem is anologous to the Cauchy integration formula in complex analysis, which says the contour integral
	\[
	\frac{1}{2\pi i}\oint \frac{f(\zeta)}{ \zeta-z  }d\zeta
	\]
	vanishes for $z$ in the exterior region of the contour, and reproduces the holomorphic function $f$ in the interior domain.
\end{rmk}

\begin{rmk}
Viewed in another way, the reconstrction theorem means there is a canonical embedding of $E$ inside the trivial Hilbert bundle $H^2_D\otimes S_-$. Then (\ref{Greenfunctionembeddingdata}) has the interesting interpretation:
\begin{cor}
The Green's function $G_0(x,y)$ on the original bundle is related to the embedding data by
\begin{equation}
G_0(x,y)=\frac{1}{4\pi^2|x-y|^2} \langle v_x(0), v_y(0) \rangle_{H^2_D \otimes S_-} \in \Hom(E_y, E_x).
\end{equation}
\end{cor}
Similar results are known in the context of the usual ADHM construction (\cf equation (29) in \cite{Christ}).
\end{rmk}

\begin{rmk}
Theorem \ref{reconstructiontheorem} links ASD connections to operator theory via the ADHM interpretation. Donaldson proved in \cite{DonaldsonRungeapproximation} a Runge approximation type theorem, which says that ASD connections on domains can be $C^\infty$ approximated on compact subsets up to gauge, by restrictions of global ASD connections over $S^4$ for arbitrarily large second Chern class. This suggests that on the operator theory side, one may approximate $\hat{x}_\mu$ in some sense by finite rank Hermitian operators satisfying the ADHM equation. It is interesting to ask how this may be proved using purely operator theoretic techniques.
\end{rmk}

\begin{rmk}
The ADHM data $(H^2_D, \hat{x}_\mu)$ resembles the key concept of `spectral triple' in noncommutative geometry (NG) \cite{AlainConnes}, and the ADHM type equation (\ref{ADHMequation}) fits into the basic philosophy of NG, namely to encode geometry by operator algebras.
\end{rmk}

\section{The spectral problem}\label{ThespectralproblemChapter}

For a smooth ASD connection $A$ on $B=B(R)\subset \R^4$, we wish to understand the Bergmann space $H^2_D =\{ s\in L^2(B, E\otimes S_- ): D_A^-  s=0           \}      $, whose elements are called Dirac fields, by studying the spectrum of a natural operator on $H^2_D$, which physically is just the Hamiltonian of a harmonic oscillator (\cf Section \ref{Thespectralproblem}).

We are interested in the behaviour of $H^2_{D_A}$ as $A$ develops a curvature singularity at the origin. This means $A$ is a member of a sequence $A_i$ (or a 1-parameter family $A_t$), which is uniformly bounded to all orders on the complement of any given neighbourhood of the origin. We derive uniform estimates to control the eigenstates associated to the spectral problem. These are based on the Weitzenb\"ock formula. The picture emerging from the analysis is the following:

The spectrum is divided into 3 characteristic ranges: $\lambda<<1$, the intermediate range, and $\lambda \sim 1$. 
\begin{itemize}
	\item If the eigenvalue is bounded below by a positive constant, then the solution has good interior Morrey type bound. (\cf (\ref{interiorcontrolforlargeeigenvalue}))
	\item If the eigenvalue is bounded above away from 1, then the solution has good bounds away from the origin (\cf Proposition \ref{smoothcontrolawayfromorigin} and its ensuing Remark).
	\item If the eigenvalue is small, then the density of the eigenstate is concentrated near the origin (\cf (\ref{concentrationnearorigin})).
	\item If the eigenvalue is close to 1, then the density of the eigenstate is concentrated near the boundary (\cf (\ref{concentrationnearboundary})).
\end{itemize}

\subsection{The spectral problem}\label{Thespectralproblem}

Let $A$ be an ASD connection on $B=B(R)$, which is smooth up to the boundary. We introduce a functional on the Bergmann space $H^2_D$,
\begin{equation}
H(s)=\frac{1}{2R^2} \int_{B(R)} |x|^2 |s|^2 d\text{Vol}.
\end{equation}
Physically this is the Hamiltonian of a harmonic oscillator.

Now $H$ defines a Hermitian form, which in the presence of the $L^2$ inner product defines a bounded self-adjoint operator $L$ acting on $H^2_D$. 
We observe

\begin{lem}
	The operator $L=1/2+K$ where $K$ is a compact operator.
\end{lem}

\begin{proof}
	We can write $L-1/2$ as the operator norm limit of a sequence of linear operators $L_i$ corresponding to the Hermitian forms 
	\[
	H_i(s)=\frac{1}{2}\int_{B((1-1/i)R)} (|x|^2R^{-2}-1)|s|^2
	\]
	But each $L_i$ is a compact operator becasuse the $L^2$ norm of a Dirac field controlls all interior higher order derivatives. Compact operators are closed under norm limits, so $L-1/2$ must be compact as well.
\end{proof}

Standard functional analysis implies that $L$ has \textbf{discrete spectrum}; this amounts to the simultaneous diagonalisation of the the Hermitian form $H$ and the $L^2$ inner product, and is computable by the Rayleigh-Ritz method.

\begin{rmk}
	We have by definition
	\begin{equation} 
	\int_{B} \frac{|x|^2}{R^2} |s|^2= \lambda^2 \int_{B}  |s|^2
	\end{equation}
	so $0<\lambda<1$. We make an elementary observation that when $\lambda<<1$, the eigenstate $s$ is concentrated near the origin: for any fixed $0<r<R$,
	\begin{equation}\label{concentrationnearorigin}
	\int_{B\setminus B(r)} |s|^2 \leq \frac{\lambda^2 R^2}{r^2} \int_B |s|^2.
	\end{equation}
	Heuristically, the small eigenvalue eigenstates describe particles trapped in the potential well, analogous to  bound states in physics. If $A$ has very concentrated curvature, we expect the characteristic length scale $\lambda$ of the eigenstates corresponding to the small eigenvalues to be roughly the same as the length scale of the curvature of $A$.
	
	On the opposite extreme, if the eigenvalue is close to 1, then the eigenstate $s$ is concentrated near the boundary: for any fixed $0<r<R$,
	\begin{equation}\label{concentrationnearboundary}
	\int_{B(r)} |s|^2 \leq \frac{(1-\lambda^2)R^2}{R^2-r^2} \int_B |s|^2 
	\end{equation}
\end{rmk}

We now characterise the eigenstates $s$ of $L$:

\begin{prop} (Characterisation of eigenstates)
	There is a unique section $\zeta$ in  $W^{2,1}_0(B, E\otimes S_+)$, such that
	\begin{equation}\label{characterisationofeigenstates}
	\begin{cases}
	D^-_A s=0  \\
	D^+_A\zeta=(\frac{|x|^2}{2R^2} -\frac{\lambda^2}{2})  s
	\end{cases}
	\end{equation}
\end{prop}

\begin{proof}
	If $s\in H^2_D$ is an eigenstate of $L$ with eigenvalue $\frac{\lambda^2}{2}$, then $(\frac{|x|^2}{2R^2} -\frac{\lambda^2}{2})  s$ is orthogonal to all elements in the space $H^2_D$ in the $L^2$ sense, by linear algebra. This implies the second equation above by using the decomposition
	\[
	L^2(B, E\otimes S_-)=  D^+_A W^{2,1}_0(B, E\otimes S_+)  \oplus H^2_D.
	\]
	Here $\zeta$ is unique because by a standard Weitzenb\"ock formula, there is no coupled Dirac field with positive spin and zero boundary condition.
\end{proof}

\subsection{The Weitzenb\"ock formula}

The well known Weitzenb\"ock formula says that because $A$ is ASD,  \[
D^2=D^-_A D^+_A=\Lap_A
\]
acting on the coupled positive spinor $\zeta$; the curvature effect is not directly visible. Thus $\zeta$ 
enjoys more favourable analytic properties compared to the negative spinors, so we shall mainly focus on $\zeta$ when deriving the estimates.

Notice equation (\ref{characterisationofeigenstates}) implies 
\begin{equation*}
D^2 \zeta=\sum_i \frac{x_i c_i}{R^2} s= \frac{x\cdot s}{R^2} ,
\end{equation*}
which, by an application of Weitzenb\"ock formula, gives
\begin{equation}\label{differentialidentityWeitzenbock}
-\Lap |\zeta|^2=2 |\nabla \zeta|^2- 2 \text{Re} \langle \zeta, \frac{x\cdot s}{R^2}  \rangle,
\end{equation}
where our convention of the Hodge Laplacian is $\Lap=-\sum_i \nabla_i\nabla_i$. 
Integrating the identity, we obtain:
\begin{equation*}\label{gradientL2zeta1}
\int_{B}  |\nabla \zeta|^2=\text{Re} \int_{B} \langle \zeta, \frac{x\cdot s}{R^2}  \rangle
\end{equation*}
Here the boundary term does not appear because $|\zeta|^2$ vanishes to second order.

\begin{lem} ($W^{2,1}$ estimate of $\zeta$)
	\begin{equation}\label{gradientL2zeta}
	\int_{B}  |\nabla \zeta|^2\leq C\lambda^2 \int_{B} |s|^2.
	\end{equation}
	As a consequence,
	\begin{equation}\label{ConsequenceofHardy}
	\int_{B} \frac{|\zeta|^2(x)}{ |x|^2  }\leq C {\lambda^2} \int_{B} |s|^2.
	\end{equation}
	Here $C$ is an absolute constant.
\end{lem}

\begin{proof}
	By the Poincar\'e inequality,
	\[
	\begin{split}
	\int_{B}  |\nabla \zeta|^2& =\text{Re} \int_{B} \langle \zeta, \frac{x\cdot s}{R^2}  \rangle \leq \frac{1}{R^2}(\int_B |\zeta|^2)^{1/2}( \int_B |x|^2|s|^2  )^{1/2} \\
	& \leq CR^{-1}(\int_B |\nabla\zeta|^2)^{1/2}( \int_B |x|^2|s|^2  )^{1/2},
	\end{split}
	\]
	so (\ref{gradientL2zeta}) follows by \[
	\int_{B}  |\nabla \zeta|^2 \leq CR^{-2} \int_B |x|^2|s|^2 = C\lambda^2 \int_{B}  |s|^2.
	\] 
	The inequality (\ref{ConsequenceofHardy}) follows from Hardy's inequality, by noticing that $\zeta$ has zero boundary condition.
\end{proof}

\begin{prop}\label{smoothcontrolawayfromorigin}
	Given constants $0<r<R$ and $0<C''<1$, in the annulus region $r\leq|x|\leq R$, if we normalise $s$ to have unit $L^2$ norm, then for eigenstates with 
	eigenvalues satisfying
	$
	\lambda< C'',
	$
	we have smooth estimates on $s$ and $\zeta$ to all orders, which are uniform in $A$ in the setup of this Chapter.
\end{prop}

\begin{proof}
	We can rewrite (\ref{characterisationofeigenstates}) as
	\[
	D_A( \frac{1}{|x|^2-\lambda^2 R^2} D_A\zeta        )=0.
	\]
	This equation is elliptic with uniformly bounded coefficients away from the locus $\{ |x|=\lambda \}$ and the origin, so by the $L^2$ estimate and the zero boundary condition on $\zeta$ , we obtain uniform smooth estimates of $\zeta$ in such region.
	
	Near the locus $\{ |x|=\lambda \}$ we can use the local elliptic regularity of the Dirac equation to estimate $s$ to all orders, which implies the uniform elliptic estimates on $\zeta$ to all orders.
\end{proof}

\begin{rmk}
	If we drop the condition $\lambda<C''$, the arguments above still imply uniform smooth bounds on $s$ away from both the origin and $\partial B$. 
\end{rmk}

\subsection{Large eigenvalues imply interior control}

We study the situation where the eigenvalue is positively bounded below.
\begin{equation}\label{largeeigenvaluecondition}
\lambda> C'>0.
\end{equation}

\begin{lem} 
	We have the estimate
	\begin{equation}
	\int_{B(R)}\frac{1}{|x|^2} |\nabla \zeta|^2(x)d\text{Vol}(x) \leq \frac{C}{R^2}  \int_{B} |s|^2.
	\end{equation}
	where $C$ is an absolute constant.
\end{lem}

\begin{proof}
	We interpret (\ref{differentialidentityWeitzenbock}) as a Poisson equation for $|\zeta|^2$, with zero boundary data. The Green representation formula gives
	\begin{equation*}
	|\zeta|^2(0)=\frac{1}{4\pi^2}\int_{B(R)} ( \frac{1}{R^2}-\frac{1}{|x|^2}  )\{
	2|\nabla \zeta|^2-2 \text{Re}\langle \zeta, \frac{x\cdot s}{R^2} \rangle
	\}  d\text{Vol}(x)
	\end{equation*}
	By rearranging terms and applying Cauchy-Schwartz inequality,
	\[
	|\zeta|^2(0)+ \frac{1}{4\pi^2}\int_B ( -\frac{1}{R^2}+\frac{1}{|x|^2}  )|\nabla \zeta|^2 \leq \frac{C}{R^2}\int_B |x|^{-1}|s||\zeta|\leq \frac{C}{R^2}(\int_B |s|^2)^{1/2} (\int_B \frac{|\zeta|^2}{|x|^2})^{1/2}.
	\]
	Now we use the consequence of Hardy's inequality (\ref{ConsequenceofHardy}) to bound the RHS by an absolute constant. We then use the $L^2$ gradient estimate on $\zeta$ (\cf (\ref{gradientL2zeta})) to drop $\int_B \frac{1}{R^2} |\nabla \zeta|^2$ from LHS, and drop the term involving $|\zeta|^2(0)$ to see the claim.
\end{proof}

\begin{prop}
	(Interior Morrey estimate for large eigenvalues) We have
	\begin{equation}\label{interiorcontrolforlargeeigenvalue}
	\int_{B}\frac{1}{|x|^2} |s|^2(x)d\text{Vol}(x) \leq \frac{C}{\lambda^2 R^2}  \int_{B} |s|^2.
	\end{equation}
	where $C$ is an absolute constant. In particular, if $\lambda$ is bounded positively below, then for any $0<r<R$, we have the Morrey decay estimate
	\begin{equation}
	\int_{B(r)} |s|^2 \leq \frac{Cr^2}{ R^2}  \int_{B} |s|^2.
	\end{equation}
\end{prop}
\begin{proof}
	We apply (\ref{characterisationofeigenstates}) and the above Lemma to see
	\[
	\int_B (\frac{\lambda^2}{2}-\frac{|x|^2}{R^2}   ) \frac{|s|^2}{|x|^2} d\text{Vol}\leq \frac{C}{R^2}  \int_{B} |s|^2.
	\]
	The result is then clear.
\end{proof}

\begin{rmk}
	This estimate, ultimately due to the Weitzenb\"ock formula, is highly \textbf{non-perturbative} because it holds even when there can be an arbitrarily large amount of curvature concentrating to the origin.
	The intuition is that if $\lambda$ is bounded below by some positive constant, then the characteristic length scale of the Dirac field is much larger than the scale of the concentrated curvature, hence the curvature singularity is not very visible to the solution. 
\end{rmk}

\section{Convergence theory for Bergmann spaces}\label{ConvergencetheoryforBergmannspaces}

Given a sequence $A_i$ or a one-paramter family $\{A_t\}_{t>0}$ of smooth ASD connections on $\bar{B}$ forming a curvature singularity at the origin (\cf the setup in Chapter \ref{ThespectralproblemChapter}). Away from the origin, we assume the connections converge in $C^\infty_{\text{loc}}$ to a connection $A_\infty$ on $E|_{B\setminus \{0\}}$, which is necessarily smooth and ASD, so by the removable singularity theorem $(E|_{B\setminus \{0\}}, A_\infty)$ extends to $(\tilde{E}, A_\infty)$. Here $\tilde{E}$ is conceptually a different topological bundle from $E$, although their $L^2$ sections can be identified. We may assume $A_\infty$ has small $L^2$ curvature, by possibly shrinking $B$. The \textbf{convergence problem} asks for a limiting description of the corresponding Bergmann spaces $H^2_{D_A}$.

The basic picture is that the part of the spectrum for $A$ above a threshold value converges to the spectrum for $A_\infty$ (\cf  Theorem \ref{Quantitativespectralgap}, Proposition \ref{spectralconvergence}); the 1-parameter family of Bergmann spaces $H^2_{D_{A_t}}$ converge in a natural way to $H^2_{D_{A_\infty}}\oplus V$ where $V$ is a finite dimensional space (\cf Theorem \ref{limitBergmannspace}); the natural operators on Bergmann spaces extend naturally to the limit space (\cf Proposition \ref{limitToeplitzoperator}, \ref{limitGreenoperator}).

\begin{rmk}
If we think of the Dirac equation as the analogue of the $\bar{\partial}$ equation, then this convergence theory is somewhat analogous to the picture in \cite{DonaldsonSun2}.
\end{rmk}

\subsection{Convergence of eigenstates}\label{Convergenceofeigenstates}

We consider a sequence of eigenstates $s_i$ associated to the connections $A_i$, with $L^2$ norms equal to 1, the eigenvalues $0<\lambda_i<1$ and corresponding coupled  positive spinor fields $\zeta_i$ solving (\ref{characterisationofeigenstates}), and we ask for a convergence theory of $s_i$. Without loss of generality $\lambda_i\to \lambda_\infty$ converges, after taking subsequence. The basic picture of Chapter \ref{ThespectralproblemChapter} implies:
\begin{itemize}
	\item If $\lambda_i\to 0$, then the density of $s_i$ is concentrated to the origin, so $s_i$ converges to zero weakly.
	\item If $\lambda_i\to 1$, then the density of $s_i$ is concentrated to the boundary, so $s_i$ converges to zero weakly.	
\end{itemize}

Now let us assume the uniform two sided eigenvalue bound $0<C'<\lambda_i< C''<1$. By the main results of Chapter \ref{ThespectralproblemChapter}, in any given annulus region $0<r\leq |x|\leq R$, we have uniform smooth estimates on all the data $\lambda_i, s_i, \zeta_i$, so we can extract a subsequence to ensure smooth convergence, by standard compactness arguments. A standard diagonal argument implies we can assume \textbf{$C^\infty_{loc}$ convergence away from the origin}, to the limiting data $\lambda_\infty$, $s_\infty$ and $\zeta_\infty$, which satisfy the limiting version of (\ref{characterisationofeigenstates}) on $B\setminus \{0\}$.

Since we have uniform estimates $\norm{s_i}_{L^2}=1$ and $\norm{\zeta_i}_{W^{2,1}}\leq C$ (\cf (\ref{gradientL2zeta})), and the norms cannot increase in the limit, so the same estimates hold for $s_\infty$ and $\zeta_{\infty}$.
By elliptic regularity of this limiting PDE system, the limiting data extend smoothly across the origin to give a solution of (\ref{characterisationofeigenstates}).

Crucially, we claim \textbf{strong convergence} of $s_i$ to $s_\infty$ inside $L^2$. The only possible issue is to lose $L^2$ mass to the origin. But this cannot happen thanks to the uniform interior Morrey estimate (\ref{interiorcontrolforlargeeigenvalue}).

In particular, the norm of $s_\infty$ does not collapse to zero, so $s_\infty\in H^2_{D_{A_\infty}}$ is an eigenstate with eigenvalue $\lambda_\infty$.

\subsection{The spectral gap}

A rather striking consequence of Section \ref{Convergenceofeigenstates} is
\begin{cor}
	(Spectral gap) Suppose the minimal eigenvalue for the limiting connection $A_\infty$ is $\lambda_0>0$. Then either $\lambda_\infty=0$ or $\lambda_\infty\geq \lambda_0$. 
\end{cor} 
\begin{proof}
	If the limit of eigenvalues $\lambda_\infty\neq 0$, then we can assume a positive lower bound on $\lambda_i$. Unless $\lambda_\infty=1$, we can also assume an upper bound smaller than 1. So we are in the situation above and we see the result from the good convergence theory.
\end{proof}

\begin{lem}
	In the special case that $A_\infty$ is the trivial flat connection, the minimal eigenvalue is $\frac{1}{2}\lambda^2=\frac{1}{3}$,  attained precisely for 
	for parallel Dirac fields.
\end{lem}
\begin{proof}
	Finding the minimal eigenvalue is the same as minimising the functional 
	\[
	H(s)=\int_{B}  \frac{|x|^2}{2R^2} |s|^2 
	\]
	subject to $Ds=0$ and $\norm{s}_{L^2}=1$.

	But we know by Weitzenb\"ock formula that $\nabla^*\nabla s=0$, so $|s|^2$ is subharmonic, with forcing term given by $2|\nabla s|^2$. This means the spherical average
	\[
	r^{-3} \int_{ \partial B(r)} |s|^2
	\]
	is a non-decreasing function in $r$. From this it is clear that the only way to minimise the functional is for the spherical averages to be constant, which implies $|\nabla s|=0$.
\end{proof}

We may also record a quantitative version of the spectral gap:

\begin{thm}\label{Quantitativespectralgap}
	(Quantitative \textbf{spectral gap}) 
	Let $A$ be an ASD connection on $\bar{B}=\overline{ B(R)}$ with fixed local smooth bounds on any compact set away from the origin. For 	
	given small numbers $\delta_1$, $\delta_2$, there are small constants $\epsilon$ and $\Lambda$, such that if $A$ satisifies
	\[
	\int_{B\setminus B(\Lambda R)}|F|^2 \leq \epsilon,\]
	then any eigenvalue $\frac{1}{2}\lambda^2$ must satisfy the dichotomy
	\[
	0<\lambda<\delta_1, \text{ or } \lambda>\sqrt{2/3}-\delta_2.
	\] 
\end{thm}

\begin{proof}
	This is the rigid version of the previous spectral gap result. The proof is a compactness argument. More precisely, assume a counterexample sequence $(A_i, s_i, \lambda_i)$  on $B$, then standard compactness theory implies $A_i\to A_\infty$ on a shrinked punctured disk $B\setminus \{0\}$. Here $A_\infty$ must be flat because its $L^2$ curvature vanishes. By the previous discussions, after taking subsequence $\lambda_i\to \lambda_\infty$ with $\lambda_\infty=0$ or $\lambda_\infty\geq \sqrt{2/3}$, which would give a contradiction if $\lambda_i$ fails the dichotomy.
\end{proof}

\begin{rmk}
	It is curious what the physical interpretation of the spectral gap should be.
\end{rmk}

Since we assume in this Chapter that $A_\infty$ has small $L^2$ curvature, the spectral gap theorem suggests us to separate the spectrum associated to $A_i$ into two parts: the \textbf{large spectrum} with $\lambda> \sqrt{2/3}-\delta_2$, and the \textbf{small spectrum} with $\lambda<\delta_1$. Correspondingly, the Bergmann spaces decompose as
\[
H^2_{D_{A_i}}=( H^2_{D_{A_i}})_{large} \oplus (H^2_{D_{A_i}})_{small}.
\]

\subsection{Convergence in the large spectrum}

The convergence theory for eigenstates can be rather formally extended to a convergence theory for the large spectrum.
For any given $A_i$, let $s^{j}_i$ be an orthonormal basis of eigenstates belonging to the large spectrum, where $j$ is arranged in increasing order of eigenvalues. We associate the data $\zeta^{j}_i$ and $\lambda^{j}_i$ in a self explanatory way. We will show

\begin{prop}\label{spectralconvergence}
	(\textbf{Spectral convergence}) The large spectrum for $A_i$ converges  to the spectrum of the smooth limit $A_\infty$, \ie if we index the eigenvalues for $A_\infty$ in increasing order as $\lambda_0^k$, then 
	\begin{equation}
	\lambda^k_0=\lim_{i\to \infty} \lambda^k_i.
	\end{equation}
\end{prop}

\begin{rmk}
	We will henceforth often suppress mentioning taking subsequences, and we shall tacitly use diagonal arguments. A posteriori we shall see that this is not necessary due to the uniqueness of limit.
\end{rmk}

We first introduce an \textbf{algorithm}. By previous work in this Chapter, either $\lambda^1_i\to 1$, or they satisfy two sided bounds so that $s^1_i$ converges strongly to $s^1_\infty$, which is some eigenstate in $H^2_{A_\infty}$ with eigenvalue $\lambda^1_\infty$. In the first case, we terminate and define $\lambda^k_\infty=1$ for all $k$.  In the second case,
we proceed with $\lambda^2_i$. We either terminate after a finite stage (which a posteriori does not happen), or continue indefinitely to achieve a sequence of limiting eigenstates $s^k_\infty$ with eigenvalues $\lambda^k_\infty$, which must be orthonormal by strong convergence.
The algorithm implies

\begin{lem}
	(generalised spectral gap) \begin{equation}
	\lambda_\infty^k \geq \lambda_0^k.
	\end{equation}
	The equality is achieved precisely if every eigenvalue $\lambda^k_0$ for $A_\infty$ arises as subsequential limits of eigenvalues, including multiplicity.
\end{lem}

Our next step is to show

\begin{lem}
	If $s\in H^2_{D_{A_\infty}}$, then it is a strong $L^2$ limit of sections of $H^2_{D_{A_i}}$.
\end{lem}

\begin{proof}
	As preparation, we study the decomposition of $L^2(E\otimes S_-)=D^+_{A_i}W^{2,1}_0\oplus H^2_{D_{A_i}}$ for varying connections $A_i$. Let $s\in L^2(E\otimes S_-)$ be a smooth section, with unique decomposition
	\begin{equation}\label{L2decompositionfors}
	s=D_{A_i}\tau_i+\sigma_i.
	\end{equation}
	Notice by orthogonality
	\begin{equation}\label{L2decompositionnorm}
	\norm{s}_{L^2}^2=\norm{D_{A_i}\tau_i}_{L^2}^2+\norm{\sigma_i}_{L^2}^2.
	\end{equation}
	Now since $\tau $ has positive spin and zero boundary condition, using Weitzenb\"ock, one has
	\[
	\norm{D_{A_i}\tau}_{L^2}^2=\norm{\nabla_{A_i} \tau_i}_{L^2}^2\sim\norm{ \tau_i}_{W^{2,1}_0}^2.
	\]
	These preliminary remarks establish that the $L^2$ decomposition in this context respects the natural norms.

	By the norm control, on every annulus $0<r\leq |x|\leq R$ we can extract smoothly convergent subsequences for $\tau_i$ and $\sigma_i$. Hence we have smooth convergence on $0<|x|\leq R$ to the limiting data $(\tau_\infty, \sigma_\infty)$, for which the limiting version of (\ref{L2decompositionfors}) holds, and morever $\norm{\tau_\infty}_{W^{2,1}_0}\leq C$, $\norm{\sigma_\infty}_{L^2}\leq C$. Notice the equation
	\[
	D_{A_\infty}^2\tau_\infty=D_{A_\infty} s
	\]
	implies that $\tau_\infty$ extends to a smooth section, so $\sigma_{\infty}$ is also smooth. It is clear that the norm identity (\ref{L2decompositionnorm}) holds in the limit, namely
	\[
	\norm{s}_{L^2}^2=\norm{\nabla_{A_\infty}\tau_\infty}_{L^2}^2+\norm{\sigma_\infty}_{L^2}^2.
	\]
	But Fatou's lemma implies that 
	\[
	\liminf \norm{\tau_i}^2 \geq \norm{\tau_\infty}^2
	, \quad
	\liminf \norm{\sigma_i}^2 \geq \norm{\sigma_\infty}^2
	\]
	So the only possibility is for equalities to be achieved everywhere, \ie
	\[
	\lim \norm{\tau_i}^2 = \norm{\tau_\infty}^2,\quad 
	\lim \norm{\sigma_i}^2 =\norm{\sigma_\infty}^2.
	\]
	The non-collapsing of norms imply that $\sigma_i$ converges strongly to $\sigma_\infty$ in $L^2$. Notice also that since the subsequential limit is unique, in fact the whole sequence has to converge. 
	
	Now by an approximation argument on $s$, it is clear that the smoothness of $s$ is not essential. We specialise to the case $s\in H^2_{D_{A_\infty}}$. Then $s=\sigma_\infty$ and the claim follows.
\end{proof}

\begin{cor}
	If morever $s$ is an eigenstate for $A_\infty$, for which the eigenvalue $\lambda_0^m$ has multiplicity one, then we may assume the sequence consists of eigenstates as well. If the eigenvalue is degenerate, we need to take linear combinations of eigenstates with approximately the same eigenvalue. In particular $\lambda_0^m$ arises as a limit of eigenvalues.
\end{cor}

\begin{proof}
	We argue in the nondegenerate case (the degenerate case has only a little more combinatorial complexity). The crucial point is that the spectrum is discrete. There are only finitely many eigenvalues below $\lambda^m_0$:
	\[ \lambda^1_0 \leq \lambda^2_0 \ldots \leq \lambda^{m-1}_0 <\lambda_0^m <\lambda^{m+1}_0 \leq \ldots \]
	Consider the spectral decomposition 
	\[
	\sigma_i=\sum_j a_j s_i^j + (\text{small spectrum contribution}) 
	\]
	By the generalised spectral gap lemma, there are essentially at most $m-1$ eigenvalues $\lambda_i^k$ bounded above by $\lambda_0^{m-1}+\epsilon<\lambda_0^m.$ For these eigenvalues the eigenstates are almost orthogonal to $s$ when $i$ is large, because they converge to eigenstates with lower eigenvalues. So their corresponding Fourier coeffients $a_j$ in the spectral decomposition must go to zero. But once we are not allowed to have contributions from low eigenvalues, then for overall $L^2$ mass reasons, neither are we allowed to have contributions from eigenvalues larger than $\lambda^{m+1}_0-\epsilon> \lambda^m_0$.
\end{proof}

%\begin{rmk}
	%The fact that the degenerate spectrum can split is well known in perturbation theory.
%\end{rmk}

Once we know $\lambda^k_0$ arises as limits of eigenvalues, Proposition \ref{spectralconvergence} follows from the generalised spectral gap Lemma. We leave the reader to ponder the issue of multiplicity. As remarked earlier, the limit of $\lambda_i^k$ turns out a posteriori to be independent of the subsequence.

%\begin{rmk}
%Modulo the issue of degenerate eigenvalues, the subsequential limits are unique, so in fact we don't need to take subsequence in the first place.
%\end{rmk}

\subsection{The limit of the Bergmann spaces}

From last section, one has the appealing picture that $H^2_{D_{A_\infty}}$ is the limit of $(H^2_{D_{A_i}})_{large}$, while the small spectrum is lost in the na\"ive smooth limit. 
In this Section we give a more operator theoretic perspective; the main result is Proposition \ref{limitBergmannspace}.

The first step is to study a comparison map between $(H^2_{D_{A_i}})_{large}$ and $H^2_{D_{A_\infty}}$. To save some writing, we will often suppress sequential indices $i$. Recall $D_A$ and $D_{A_\infty}$ induce two decompositions of $L^2$. In particular, the two Bergmann spaces $H^2_{D_A}$ and $H^2_{D_{A_\infty}}$  project to each other, via the operator $\mathcal{P}:H^2_{D_{A_\infty}}   
\rightarrow H^2_{D_{A}}$ and its adjoint $\mathcal{P}^\dagger$. Thus there are canonical maps 
\begin{equation}\label{positivespectrumcomparisonmap}
\pi \mathcal{P} :H^2_{D_{A_\infty}}
\rightarrow H^2_{D_A} \xrightarrow{\text{project} } (H^2_{D_A})_{large}
\end{equation}
and its adjoint
\begin{equation}\label{positivespectrumcomparisonmap2}
(H^2_{D_{A}})_{large} \rightarrow H^2_{D_A} \rightarrow  H^2_{D_{A_\infty}}.
\end{equation}
%These are clearly Fredholm, because $\mathcal{P}$ and $\mathcal{P}^\dagger$ are compact perturbations of the identity operator, if we view the Bergmann spaces inside $L^2$.

\begin{lem}
	The canonical map $\pi \mathcal{P}: H^2_{D_{A_\infty}} \rightarrow (H^2_{D_{A_i}})_{large} \subset L^2$  converges to the identity operator $I_{ H^2_{D_{A_\infty}}  }$ in the operator norm, as $i\to \infty$.
\end{lem}

\begin{proof}
	Pick any $s\in H^2_{D_{A_\infty}}$, which is normalised to $\norm{s}_{L^2}=1$. We have the decomposition 
	\begin{equation}
	s=D_{A}\tau+\sigma+\sigma'
	\end{equation}
	where $\tau\in W^{2,1}_0$, $\sigma \in (H^2_{D_A})_{large}$ and $\sigma'\in (H^2_{D_A})_{small}$. Here $\sigma$ is the image of the canonical map. Then 
	\begin{equation}\label{partitionofnorm}
	\norm{\nabla_{A} \tau}_{L^2}^2+\norm{\sigma}^2_{L^2}+\norm{\sigma'}_{L^2}^2=\norm{s}_{L^2}^2.
	\end{equation}
	In particular, $\tau$ is bounded in $L^2$. We also have
	\[
	D_A^2\tau=D_A s=(D_A-D_{A_\infty})s,
	\]
	so $D_A^2\tau$ is $L^2$ small away from the origin, whereby $\tau$ is $W^{2,2}_{loc}(B\setminus \{0\})$ bounded. We can therefore assume $\tau$ to converge strongly in $W^{2,1}_{loc}(B\setminus\{0\})$ as $i\to \infty$. The limit $\tau_\infty$ is globally bounded in $L^2$, has zero boundary condition, and satisfies the Laplace equation, so must be zero. Thus in the sequence,  $\tau$ must be $W^{2,1}$ close to $0$ away from the origin, with bounds independent of $s$. In particular, $D_A \tau$ is $L^2$ small away from the origin. 
	
	Notice also 
	\[
	\int_{B}|x|^2R^{-2}|\sigma'|^2<<\int_{B}|\sigma'|^{2}\leq 1.
	\]
	Therefore $s$ is $L^2$ close to $\sigma$ away from the origin. In particular, since the Dirac equation implies $s$ has pointwise bound near the origin, we have for fixed $\delta<<1$,
	\[
	\int_{B(R)\setminus B(\delta R)}|s|^2 \geq (1-C\delta^4) \int_{B(R)} |s|^2,
	\]
	so by the above argument, when $i$ is sufficiently large depending only on $\delta$,
	\[
	\int_{B(R)\setminus B(\delta R)} |\sigma|^2 \geq (1-C\delta^4) \int_{B(R)} |s|^2.
	\]
	Comparing with the partition of norm (\ref{partitionofnorm}), 
	\[
	\norm{D_A\tau}_{L^2} \leq C\delta^2 \norm{s}_{L^2}, \quad \norm{\sigma'}_{L^2}\leq C\delta^2 \norm{s}_{L^2}.
	\]
	so $
	\norm{s-\sigma}_{L^2} \leq C\delta^2 \norm{s}_{L^2}.
	$
	The constants are independent of $s$. Since $\delta$ is arbitrary, this proves the result.
\end{proof}

\begin{lem}
	The canonical comparison map $\pi\mathcal{P}$ for $A_i$ is surjective for $i>>1$.
\end{lem}
\begin{proof}
	Since the above Lemma implies the coercivity of $\pi \mathcal{P}$, the image is closed.
	The cokernel of $\pi\mathcal{P}$ is the kernel of (\ref{positivespectrumcomparisonmap2}), which means $s\in (H^2_{D_A})_{large}$, and $s=D_{A_\infty}\tau$, for $\tau \in W^{2,1}_0$. If the cokernel does not vanish for a subsequence of connections $A_i$, we normalise $s$ to $\norm{s}_{L^2}=1$, and derive a contradiction by a compactness argument as follows.
	
	First, $A_\infty$ is ASD, so $\norm{\nabla_{A_\infty} \tau}_{L^2}=\norm{D_{A_\infty} \tau}_{L^2}=1$, hence $\norm{\tau}_{L^2}$ is controlled. Now $\tau$ satisfies $D_{A_i} D_{A_\infty} \tau=0$, with the zero boundary condition, hence it is controlled to any order away from the origin. Any subsequential limit of $\tau$ must be zero, by the same argument as in last Lemma.
	This implies $\tau$ converges to zero smoothly in the punctured disc; so must $s$, which means the $L^2$ mass of $s$ is concentrated at the origin. But $s$ lives in the large spectrum, so $\int_{B(R)} |x|^2R^{-2}|s|^2 \geq C \norm{s}_{L^2}^2$, contradiction.
\end{proof}

\begin{cor}
	For large $i$, the canonical map $\pi \mathcal{P}$ is an isomorphism. Morever, it is close to being a unitary equivalence:
	\[
	\norm{(\pi\mathcal{P})^\dag- (\pi\mathcal{P})^{-1} }\to 1,\text{ as } i\to \infty.
	\]
\end{cor}

We next study the small spectrum $H^2_{D_{A}}$. Here it is more convenient to set up the problem in terms of a \textbf{one-parameter family} ${(A_t)}_{t>0}$ of ASD connections converging smoothly away from the origin. This is more suited to continuity arguments.

\begin{prop}
	For large $t$ so that the curvature is sufficiently concentrated, the \textbf{dimension of the small spectrum} is constant in the family.
\end{prop}

\begin{proof}
	Imagine $t$ to flow from a large number to $\infty$. By general functional analysis, the eigenvalues flow continuously. But spectral gap (\cf Theorem \ref{Quantitativespectralgap}) prevents the spectral flow between the large spectrum and the small spectrum, so the dimension of the small spectrum must be a constant finite number.
\end{proof}

\begin{rmk}
	The concept of small spectrum is only meaningful in the large $t$ limit.
\end{rmk}

By examining the comparison map $\pi \mathcal{P}$ in the limit $t\to \infty$, we see

\begin{lem}
	For large $t$, the operator $\mathcal{P}: H^2_{D_{A_\infty}} \rightarrow H^2_{D_{A_t}}$ is Fredholm, with index being the negative of the dimension of the small spectrum.
\end{lem}

%\begin{rmk}
%This does not require the curvature concentration assumption.
%\end{rmk}

\begin{lem}
	The Bergmann spaces $H^2_{D_{A_t}}$ fit into a Hilbert bundle over the half line $0<t<\infty$. 
\end{lem}

\begin{proof}
	Consider $t$ near $t_0$, and the natural projection operator $H^2_{D_{A_t}}\rightarrow H^2_{D_{A_{t_0}}}$. This is an isomorphism by the argument proving $\pi\mathcal{P}$ to be an isomorphism.
\end{proof}

\begin{lem}
	For large $t<\infty$, the subspaces $(H^2_{D_{A_t}})_{small}$ fit together into a vector bundle.
\end{lem}
\begin{proof}
	As a general fact, consider a family of operators depending smoothly on a parameter $t$, with discrete spectrum consisting of eigenvalues. If we fix a spectral domain $\Omega\subset \C$, and consider the projection operator to the the span of the eigenspaces for all the eigenvalues inside $\Omega$, then as long as no eigenvalue crosses the boundary of $\Omega$, the projection operator depends smoothly on $t$, and the image has constant dimension. (\cf Appendix of \cite{Kodaira}).
\end{proof}

\begin{thm}
	(Natural limit of Bergmann spaces)\label{limitBergmannspace}
	Let $V$ be an inner product space with the dimension of the small spectrum. Then there is a natural topological bundle over $0<t\leq \infty$, whose fibres over $0<t<\infty$ agree with the Bergmann space $H^2_{D_{A_t}}$, and the fibre over $\infty$ is $H^2_{D_{A_\infty}}\oplus V$.
\end{thm}

\begin{proof}
	It suffices to assign a trivialisation near $\infty$. One can use the isomorphism provided by the canonical comparison map (\ref{positivespectrumcomparisonmap}) to deal with the large spectrum part. To trivialise the small spectrum part, one notices there is a canonical connection on the bundle $\bigsqcup_{\infty>t>t_0>>1} (H^2_{D_{A_t}})_{small}$ coming from the embedding into $L^2$. Over the one dimensional base $(t_0, \infty)$, this connection gives a parallelisation of this finite rank bundle compatible with the Hermitian structure, which can be extended to the $\infty$ fibre by formally adding a copy of $V$.   
\end{proof}

\subsection{Natural operators on the Bergmann space}

In the setup of Theorem \ref{limitBergmannspace}, we can ask whether natural operators on $H^2_{D_{A_t}}$ extend continuously to operators on the limit Bergmann space $H^2_{D_{A_\infty}}\oplus V$. Here we use the trivialisation near $\infty$ in Theorem \ref{limitBergmannspace} to regard $H^2_{D_{A_t}}$ as a fixed Hilbert space, and the convergence of operators is reduced to the usual definition of norm convergence.

Given a smooth function $f$ on $\bar{B}$, the \textbf{Toeplitz operator} $\hat{f}= P_A \circ f$ on $H^2_{D_A}$ means multiplying by the function $f$, composed with the orthogonal projection of $L^2$ to $H^2_{D_A}$. Physically, the Bergmann space is the state space of a fermion, and then these correspond to quantum observables of the shape $\langle s|f|s'\rangle$. The Toeplitz operators are Hermitian.

\begin{prop}\label{limitToeplitzoperator}
	(Limit of Toeplitz type operators)
	In the setup of Theorem \ref{limitBergmannspace}, if $f(0)=0$, then the Toeplitz operator $\hat{f}$ on the Bergmann spaces $H^2_{D_{A_t}}$ converge strongly to the operator
	\begin{equation}
	H^2_{D_{A_\infty}}\oplus V \xrightarrow{\hat{f}\oplus 0 } H^2_{D_{A_\infty}}\oplus V
	\end{equation}
	on the limit Bergmann space, where  $\hat{f}$ also denotes the corresponding Toeplitz operator on $H^2_{D_{A_\infty}}$. 
\end{prop}

\begin{proof}
	We use the trivialisation near $\infty$ described in the proof of \ref{limitBergmannspace}. The limit operator is zero on the $V$ factor, because on the small spectrum for $A_t$,
	\[
	\norm{ \widehat{f} s}^2_{L^2} \leq \int_{B(R)} f^2|s|^2 \leq C\int_{B(R)} |x|^2|s|^2<< \norm{s}^2_{L^2}.
	\]
	Hence we only need to consider the large spectrum for $A_t$. Then the result follows from the norm convergence $\pi \mathcal{P} \to 1$ as $t\to\infty$.
\end{proof}

\begin{rmk}
It is unclear to the author how to extend the matrix describing Clifford multiplications $\langle s|c_jc_k|s'\rangle$ for $s,s'\in (H^2_{D_{A_t}})_{small}$ to the limit.
This question is intimately tied to compactification of ASD moduli spaces. It is interesting to extend the local Nahm transform of Chapter \ref{ThelocalNahmtransform}, \ref{Theinverseconstruction} to singularities.  On the ASD side, the Uhlenbeck compactification involves
ideal instantons. On the operator theory side, the ideal instanton corresponds to $(H^2_{D_{A_\infty}}\oplus V, \hat{x}_\mu)$, where $H^2_{D_{A_\infty}}$ encodes the smooth limit $A_\infty$ and $\dim V$ encodes the delta mass of curvature. But if we enlarge the operator algebra to include secondary operators induced by $\langle s|c_jc_k|s'\rangle$ and ask how they converge, then we may obtain a more refined compactification.
\end{rmk}

Next we consider the Green operator $G_A: H^2_{D_A}\subset L^2 \to W^{2,1}_0\subset L^2$,  defined by solving the Laplace equation $\Lap_A s=0$ with zero boundary condition. For a point $x$ inside the annulus $0<r<|x|<r'<R$, elliptic regularity gives the pointwise estimate
\[
|G_A s(x)|\leq C\norm{G_A s}_{L^2}.
\]
Thus the evaluation map
\[
ev_x\circ G_A: H^2_{D_A} \to E_{x}\otimes S_-, \quad s\mapsto G_A s(x)
\]
is a bounded $E_x\otimes S_-$ valued linear functional.

Now we vary the connection, so the Green's operator depends on the parameter $t$, and we study its limit.

\begin{prop} (limit Green operator)\label{limitGreenoperator}
	On the annulus $0<r<|x|<r'<R$, as $t\to \infty$, the evaluation maps for the Green operators converge strongly and uniformly in $x$ to some $E_x\otimes S_-$ valued bounded linear functional on the limit Bergmann space $H^2_{D_{A_\infty}}\oplus V$. The limit functional vanishes on $V$ and agrees with $ev_x\circ G_{A_\infty}$ on the $H^2_{D_{A_\infty}}$ factor.
\end{prop}

\begin{proof}
	We first consider any $s\in (H^2_{D_{A_t}})_{small}$. Consider the $L^2$ decomposition
	\[
	x_\mu s=D_A \tau_\mu+\sigma_\mu,
	\]
	where $\tau \in W^{2,1}_0, \sigma_\mu \in H^2_D$.
	Then 
	\[
	D_A^2\tau_\mu=D_A(x_\mu s)=c_\mu s,
	\]
	so $\tau_\mu=c_\mu (Gs)$. Hence
	\begin{equation*}
	C\norm{G_A s}_{L^2}^2\leq \norm{\nabla (G_A s)}_{L^2}^2=\norm{\nabla\tau_\mu}^2_{L^2}=\norm{D\tau_\mu}^2_{L^2}\leq \norm{x_\mu s}^2_{L^2}<< \norm{s}_{L^2}^2,
	\end{equation*}
	where $C$ is an absolute constant. Thus the norm of the evaluation functional on the small spectrum converges to zero.

	On the large spectrum, we use a compactness argument. Take some $L^2$ normalised counterexample sequence $s_i \in H^2_{D_{A_{\infty}}}$, and let $s_i'=\mathcal{P}_i s_i \in H^2_{D_{A_{t_i}}}$ be their corresponding elements via the trivialisation. We have points $x_i$ inside the annulus, such that $|G_{A_{t_i}} s_i'(x_i)-G_{A_\infty}s_i(x_i)|\geq \epsilon$. We will freely pass to subsequences.
	
	Due to uniform $W^{2,1}_0$ bounds, we can extract a weak limit for $G_{A_{t_i}} s_i'$. By elliptic estimates in the annulus region, we can assume the convergence to be uniform in $0<r<|x|<r'<R$, and $s_i$ converges weakly to $s_\infty\in H^2_{D_{A_\infty}}$. Since $\mathcal{P}_i \to 1$, the sequence $s_i'$ has the same limit $s_\infty$. Then using the weak equation for the limit, one sees the weak limit of $G_{A_{t_i}} s_i'$ must be $G_{A_\infty}s_\infty$. The same discussions apply to $G_{A_\infty} s_i$. Thus the difference $ G_{A_{t_i}} s_i'-G_{A_\infty}s_i$ has weak limit zero, the convergence is uniform in the annulus, but $|G_{A_{t_i}} s_i'(x_i)-G_{A_\infty}s_i(x_i)|\geq \epsilon$, contradiction.
\end{proof}

\begin{rmk}
	These linear functionals can be Riesz represented as elements of $H^2_{D_A}\otimes (E_x\otimes S_-)$, which also converge strongly as $t\to \infty$. 
\end{rmk}

\section{Index theory and Chern numbers}

\subsection{The rank of $\hat{\hat{E}}$ and index computation}\label{Therankofdoubletransformandindexcomputation}

We prove Lemma \ref{indexproblem}, which amounts to computing the index of an operator. It is convenient to work in the ADHM formulation, which links more easily to operator theory. The discussions below are self contained, although the main results are likely to be known in the literature of Toeplitz operators and index theory.

The index problem fits into a more general picture.
Consider the operator
\[
T=\sum_i \hat{f_i}\hat{c}_i: H^2_D\otimes S_-\rightarrow H^2_D\otimes S_+,
\]
where $\hat{c}_i$ is the Clifford multiplication on the spin factor, and the Toeplitz operator $\hat{f_i}=P_0 \circ f_i$ is the projection of the multiplication by a smooth real valued function $f_i$ on $\bar{B}$, $i=1,2,3,4$. 
We put on the assumption that the quaternion valued function $f_T=f_1+i f_2+j f_3+k f_4$, which we think of as the \textbf{symbol} of the operator, does not vanish anywhere on $\partial B$. In the generality of this Section, the connection $A$ needs not to be ASD.

\begin{lem}
	The operator $T$ is Fredholm.
\end{lem}

\begin{proof}
	The operator is clearly bounded. We show the finite dimensionality of the kernel. First write down the $L^2$ decomposition
	\[
	f_i s=D \tau_i +\sigma_i,
	\]
	where $s\in H^2_D\otimes S_-$, $\tau_i \in W^{2,1}_0 \otimes S_-$, and $D \sigma_i=0$. Thus using $D s=0$,
	\[
	D^2\tau_i =\sum_j(c_j \partial_j f_i) s.
	\]
	This means by elliptic regularity, $\tau_i$ can be chosen to satisfy 
	\[
	\norm{\tau_i}_{W^{2,2}_0} \leq C \norm{s}_{L^2},
	\]
	hence 
	\[
	\norm{D \tau_i}_{W^{2,1} } \leq C \norm{s}_{L^2}.
	\]
	As a remark, for general connections $A$, we need to impose $\tau_i\in (\ker D)^\perp$ to make $\tau_i$ satisfy the required estimates.
	Combining the above,
	\begin{equation}\label{Fredholmargumentbootstrap}
	\sum f_i \hat{c_i}s= \sum \hat{c}_i D \tau_i +\sum \hat{c}_i \sigma_i=\sum \hat{c}_i D \tau_i+ Ts.
	\end{equation}
	The nonvanishing of the matrix $\sum f_i\hat{c_i}$ near the boundary is equivalent to its invertibility.
	Thus in a small neighbourhood of the boundary $B_\delta=\{x\in B: \text{dist}(x, \partial B)<\delta \}$, if $Ts=0$, then $ \norm{s}_{W^{2,1}(B_\delta)} \leq C \norm{s}_{L^2}$. 
	But we also have $\norm{s}_{W^{2,1}(B\setminus B_\delta)} \leq C \norm{s}_{L^2}$  by the interior regularity of the Dirac equation, so $\norm{s}_{W^{2,1}(B)} \leq C \norm{s}_{L^2}$, which forces the kernel to be finite dimensional by compactness.
	
	A completely symmetric argument proves the finite dimensionality of the cokernel.
	
	We also need $T$ to have closed range. It is enough to prove that for $s\in (\ker D)^\perp$, we have $\norm{s}_{L^2} \leq C\norm{Ts}_{L^2}$.
	This follows from a compactness argument similar to the above. The key is to invoke (\ref{Fredholmargumentbootstrap}), and suppose for contradiction take a sequence of $s$ to converge smoothly in the interior, and $\tau_i$ to converge strongly in $W^{2,1}_0$, such that $Ts$ converges strongly to zero.
\end{proof}

Our next aim is to describe the index of this operator. We proceed by a sequence of observations:
\begin{enumerate}
	\item The index depends only on the boundary value of $f_i$.
	
	This is because if the boundary value is zero, then $T$ is a compact operator, by the interior regularity of Dirac fields. \\

	\item The map $f_T=f_1+if_2+jf_3+kf_4: \partial B \rightarrow \mathbb{H}\setminus \{0\}$ defines a degree (recall we assume $B$ is homeomorphic to a ball), which classifies its homotopic type. Homotopic Fredholm operators define the same index. Hence the index depends on $f$ only through its degree. \\

	\item For the zero degree case, we consider the special operator $T=\hat{c}_1$, which clearly gives an isomorphism, so the index is zero. \\

	\item Behaviour under multiplication.
	
	Notice $index(T)=index(\hat{c}_1T)$, but $\hat{c}_1T$ is an endomorphism of $H^2_D\otimes S_-$, so can be composed. Observe $1, \hat{c}_1\hat{c}_2, \hat{c}_1\hat{c}_3,\hat{c}_1\hat{c}_4$ is a standard quaternionic basis. Morever, for the multiplication operators $\hat{g}$ and $\hat{g'}$ acting on the $H^2_D$ factor, where $g$ and $g'$ are real valued functions on $\bar{B}$, the composite $\hat{g}\hat{g'}$ agrees with $\hat{gg'}$ up to a compact correction. Combining these, the quaternionic multiplication of the symbol functions and the composition of the operators are related as follows:
	\[
	{f_T\cdot f_{T'} }=-f_{\hat{c}_1^{-1}(\hat{c}_1 T')(\hat{c}_1T)}
	\]
	But the degree is additive with respect to symbol multiplication, and the index is additive with respect to composition, so 
	\begin{lem}
		For a fixed connection $A$, the index is proportional to the degree.
	\end{lem}
	\item Independence of connection $A$.
	
	We wish to remove the dependence on the background connection $A$. For this, introduce the Bergmann space $H^2_{\bar{D}}$ corresponding to a trivial flat connection $\bar{A}$ on the original vector bundle $E$, which exists because $B$ is contractible. This induces another orthogonal decomposition of $L^2$. The two Bergmann spaces project onto each other, via operators $\mathcal{P}: H^2_D\rightarrow H^2_{\bar{D}}$ and its adjoint $\mathcal{P}^\dag$. Another viewpoint is that $H^2_D$ and $H^2_{\bar{D}}$ have a natural $L^2$ pairing, so induce two linear operators $\mathcal{P}$ and $\mathcal{P}^\dag$. It is clear that $\mathcal{P}$ and $\mathcal{P}^\dag$ are Fredholm; in fact, inside $L^2$ they are compact perturbations of the identity operator. For example, 
	\[
	\mathcal{P}=1-D^+_{\bar{A} }G_{\bar{A} } D^-_{\bar{A} }= 1-D^+_{\bar{A} }G_{\bar{A} } (D^-_{\bar{A} } - D^-_{{A} }  ),
	\]
	where $(D^-_{\bar{A} } - D^-_{{A} }  )$ is bounded on $L^2$, so
	$
	D^+_{\bar{A} }G_{\bar{A} } (D^-_{\bar{A} } - D^-_{{A} }  ): L^2\to L^2
	$
	is compact.
	
	We can then consider the composition
	\[
	H^2_{\bar{D}}\otimes S_-\xrightarrow{\mathcal{P}^\dag} H^2_{D}\otimes S_-\xrightarrow{T} H^2_{D}\otimes S_+ \xrightarrow{\mathcal{P}} H^2_{\bar{D}} \otimes S_+.
	\]
	This has the same index as $T$, because the index of $\mathcal{P}$ cancels with that of its adjoint. The composite operator is a compact perturbation of \[
	\sum \hat{f_i}\hat{c_i}: H^2_{\bar{D}}\otimes S_-\rightarrow H^2_{\bar{D}}\otimes S_+.\]
	This shows the index is the same as that of the corresponding problem in the flat case.\\
	\item In the case of the trivial connection, it is clear that $H^2_D$ decomposes according to the rank of the vector bundle, so the index is proportional to the rank of $E$.
\end{enumerate}

Now we treat the standard case of the trivial flat line bundle, with $T=\sum\hat{x}_\mu\hat{c}_\mu$. We assume without loss of generality that the origin is an interior point of $B$, so the symbol function $f_T$ is non-vanishing on $\partial B$, and morever $f_T: \partial B\to \mathbb{H}\setminus 0$ has degree 1.

The cokernel of $T$ vanishes, thanks to our non-singularity discussion (\cf Corollary \ref{cokernelvanishingonNahmtransform}). The condition for the kernel is 
\begin{equation}\label{kernelofToeplitz}
\sum x_\mu \hat{c}_\mu s=D\tau,
\end{equation}
where $s\in H^2_D\otimes S_-$ and $\tau\in W^{2,1}_0(E\otimes S_+)\otimes S_+$, and $E$ is the trivial flat line bundle. We observe
\[
\Delta \tau=D^2\tau=\sum c_\mu \hat{c}_\mu s.
\]
The significance of this comes from representation theory. Notice 
\[
S_-\otimes S_-\xrightarrow{\sum c_\mu\hat{c}_\mu}S_+\otimes S_+
\]
is a $Spin(4)-$equivariant map. We can decompose the representations
\[
S_-\otimes S_-=\Lambda^2 S_- \oplus \text{ (3D rep of $su(3)_-$) }, S_+\otimes S_+=\Lambda^2 S_+ \oplus \text{ (3D rep of $su(3)_+$) }.
\]
By Schur's lemma, the map factors through the one dimensional representation $\Lambda^2 S_-$, and the image lands inside the line spanned by $\eta'_1\otimes \epsilon(\eta'_1)+\eta'_2\otimes \epsilon(\eta'_2) \in \Lambda^2 S_+\subset S_+\otimes S_+$, where $\eta_1',\eta_2'$ form an orthonormal basis of $S_+$. Thus $\sum c_\mu \hat{c}_\mu s$ is a scalar function times $\eta'_1\otimes \epsilon(\eta'_1)+\eta'_2\otimes \epsilon(\eta'_2) \in S_+\otimes S_+$; therefore there is some scalar function $\rho$, such that
\begin{equation}
\tau=\rho (\eta'_1\otimes \epsilon(\eta'_1)+\eta'_2\otimes \epsilon(\eta'_2) ).
\end{equation}
This means
\[
\sum c_\mu \hat{c}_\mu s=\Delta\rho (\eta'_1\otimes \epsilon(\eta'_1)+\eta'_2\otimes \epsilon(\eta'_2) ).
\]
On the other hand, direct differentiation shows
\[
D\tau=\nabla \rho \cdot \eta_1'\otimes \epsilon(\eta_1')+\nabla \rho \cdot \eta_2'\otimes \epsilon(\eta_2'),
\]
hence using the defining equation (\ref{kernelofToeplitz}),
\[
s=-\frac{1}{|x|^2} \{   \nabla \rho \cdot \eta_1'\otimes x\cdot \epsilon(\eta_1')+\nabla \rho \cdot \eta_2'\otimes x \cdot \epsilon(\eta_2')             \},
\]
whereby 
\[
\sum c_\mu\hat{c}_\mu s=-\frac{4}{|x|^2} \langle \nabla \rho, x \rangle (\eta'_1\otimes \epsilon(\eta'_1)+\eta'_2\otimes \epsilon(\eta'_2) ).
\]
Comparing the above, $\Delta \rho=-\frac{4}{|x|^2} \langle \nabla \rho, x \rangle $, or
\[
\Delta (\frac{\rho}{|x|^2})=0.
\]
Recall that $\rho$ is a smooth function with zero boundary condition, so $\frac{\rho}{|x|^2}$ has zero boundary condition, and the only possible singularity is a pole of order 2 at the origin, which by our assumption is an interior point of $B$. It has to be proportional to the Dirichlet Green's function with delta mass placed at the origin. This proves the kernel dimension is 1, so the index is 1 in the special case.

\begin{eg}
	Consider the special case where $B=B(R)\subset \R^4$. Up to a constant, \[
	\rho=\frac{1}{2}(|x|^2-R^2), s=-(\eta_1\otimes \epsilon(\eta_1)+\eta_2\otimes \epsilon(\eta_2) ),
	\]
	where $\eta_1, \eta_2$ form an orthonormal basis of $S_-$. 
\end{eg}

To summarise the results of this Section,
\begin{prop}
	The \textbf{index} of $T$ equals $ \text{rank}(E)\deg(f_T)$.
\end{prop}

\begin{rmk}
	Compare this with the non-singularity discussion (\cf Corollary \ref{cokernelvanishingonNahmtransform}). This shows the inverse Nahm transform bundle $\hat{\hat{E}}$ has rank equal to $\text{rank}(E)$ in the interior of $B$ and vanishes in the exterior, thus resolving Lemma \ref{indexproblem}. In particular there is a jump of index when we cross the boundary, so the operator $T=\sum_{\mu}(\hat{x}_\mu-y_\mu)\hat{c}_\mu$ fails to be Fredholm when $y\in \partial B$.
\end{rmk}

\subsection{Singularity formation and small spectrum}\label{Singularityformationandsmallspectrum}

The aim of this Section is to show Theorem \ref{instantonnumberequalsdimsmallspectrum}. The key input is the following result, which combines our local Nahm transform theory with our analytic convergence theory for singularity formation.

\begin{lem}
In the setup of a 1-parameter family of ASD connections $A_t$ developing a curvature singularity at the origin,
we have the exact sequence
\begin{equation}\label{coherentsheafpresentation}
0\to E\xrightarrow{\alpha_t} H^2_{D_{A_t}}\otimes S_- \xrightarrow{\sum_{ \mu}(\hat{x}_\mu-y_\mu)\hat{c}_\mu } H^2_{D_{A_t}}\otimes S_+\to 0,
\end{equation}
In the limit $t\to \infty$, the operator $ \sum_{\mu} (\hat{x}_\mu-y_\mu)\hat{c}_\mu  $ tends to
\[
(H^2_{D_{A_\infty}}\oplus V)\otimes S_- \xrightarrow{\sum_{ \mu}(\hat{x}_\mu\oplus 0 -y_\mu)\hat{c}_\mu } (H^2_{D_{A_\infty}}\oplus V)\otimes S_+,
\]
where convergence holds on the whole ball, and $\alpha_t$ tends to 
\[
E\xrightarrow{\alpha_\infty \oplus 0} (H^2_{D_{A_\infty}}\oplus V)\otimes S_-,
\]
where convergence holds on $\{0<|x|<R\}$, and is uniform on compact subsets.
\end{lem}

\begin{proof}
The existence of the exact sequence is rephrasing the reconstruction theorem \ref{Thereconstructiontheorem}.

The convergence of $(\hat{x}_\mu-y_\mu)\hat{c}_\mu  $ follows from the convergence theory of Toeplitz operators (\cf Proposition \ref{limitToeplitzoperator}), applied to $f=x_\mu$.
The convergence of the canonical comparison maps $\alpha_t$  follows from Proposition \ref{limitGreenoperator}, because $\alpha_t$ is essentially defined as the evaluation map of the Green operator.
\end{proof}

\begin{rmk}
The heuristic idea of Theorem \ref{instantonnumberequalsdimsmallspectrum} is that, by (\ref{instantonnumberChernclass}) we can think of the instanton number as a `Chern class'. Since we have the exact sequence (\ref{coherentsheafpresentation}), 
we think of the `Chern class' of $E$ as a `difference of the Chern classes for infinite rank bundles', and compute it by topological manipulations. This idea involves several difficulties:
\begin{itemize}
	\item This involves infinite rank bundles.
	\item The domain is not a closed manifold.
	\item In the limit the exactness fails.
\end{itemize}
\end{rmk}

The main idea to remedy these, is to replace vector bundles by relative K-theory classes, and work in the Fredholm setting.

We sketch the following relative version of the index bundle construction, using only first principles in K-theory.

\begin{lem}
Let $X$ be a compact connected manifold with boundary $Y$. Given the data 
\begin{itemize}
\item A family of Fredholm maps between (possibly finite rank) Hilbert bundles, defined over $X$: \[
\mathcal{H}_1 \xrightarrow{\mathcal{F} } \mathcal{H}_2. 
\]
\item A trivial finite rank vector bundle $\tilde{E}|_Y$ over $Y$, with a morphism $\alpha$ into $\ker \mathcal{F}|_Y$, such that the complex
\[
0 \to \tilde{E}|_Y \xrightarrow{\alpha} \mathcal{H}_1|_Y \xrightarrow{\mathcal{F} } \mathcal{H}_2|_Y \to 0
\]
is a short exact sequence. 
\end{itemize}
Then we obtain a natural class in the relative K-theory $K(X,Y)$. Denote this as $ \text{Ind } [\mathcal{H}_1 \xrightarrow{\mathcal{F} } \mathcal{H}_2, \tilde{E}|_Y, \alpha    ] $.
Morever, this class is invariant under small collective deformations of all the defining data, and the construction is naturally additive.
\end{lem}

\begin{proof}
(Sketch) We follow the following steps:
\begin{enumerate}
\item Since $X$ is connected, the index of $\mathcal{F}$ over each point in $X$ is constant. So if $\mathcal{F}$ is fibrewise surjective, then the kernel bundle is a finite rank vector bundle over $X$ with a trivialisation over $Y$, so defines a relative K-theory class. Morever, the construction is invariant under small deformations of $\mathcal{F}$  away from $Y$.

\item The surjectivity of $\mathcal{F}$ is satisfied near $Y$. If it fails somewhere away from $Y$, we can replace $\mathcal{H}_1$ by $\mathcal{H}_1\oplus \C^N$, where $\C^N$ is a trivial vector bundle with sufficiently large rank. Now perturb $\mathcal{F}$ by adding a morphism $p: \C^N \rightarrow \mathcal{H}_2$.  Generically this will force surjectivity to hold. We require $p$ to vanish near $Y$; this uses the existence of local cutoff functions. We define the class by 
\[   \text{Ind } [\mathcal{H}_1 \oplus \C^N \xrightarrow{\mathcal{F}\oplus p } \mathcal{H}_2, \tilde{E}|_Y\oplus \C^N|_Y, \alpha\oplus Id_{\C^N}    ] .                       \]

\item We check the well definition of the above construction. Clearly
\[   
\begin{split}
&\text{Ind } [\mathcal{H}_1 \oplus \C^N \xrightarrow{\mathcal{F}\oplus p } \mathcal{H}_2, \tilde{E}|_Y \oplus \C^N|_Y, \alpha \oplus Id_{\C^N}   ]  \\
&=\text{Ind } [\mathcal{H}_1 \oplus \C^{N+N'} \xrightarrow{\mathcal{F}\oplus p\oplus 0 } \mathcal{H}_2, \tilde{E}|_Y\oplus \C^{N+N'}, \alpha \oplus Id_{\C^{N+N'}}   ] .   
\end{split} \]
Morever any two choices of perturbations $p$ and $p'$ are connected by some path. This path may pass through some elements $p''$ which do not make $\mathcal{F}\oplus p''$ surjective. But there is some $N'$ such that $p\oplus 0: \C^{N+N'}\to \mathcal{H}_2$ is connected to $p'\oplus 0$ by some non-singular path, because we can always use the extra degrees of freedom to remove any failure of transversality. The upshot is that by homotopy invariance the class is independent of the perturbation $p$, and is clearly independent of $N$.

\item
Suppose we have a one parameter family of the data $(\mathcal{H}_1  \xrightarrow{\mathcal{F} } \mathcal{H}_2, \tilde{E}|_Y, \alpha )$, parametrised by $t\in [0,1]$, with constraints as given in the theorem. We claim they define the same class. This is because without loss of generality $\mathcal{F}$ is surjective, and then the claim reduces to the usual homotopy invariance property. 

\item The additivity is obvious.
\end{enumerate}
\end{proof}

Back to our context, the characteristic number description (\ref{instantonnumberChernclass}) gives

\begin{lem}
The instanton number is equal to the second Chern class of the relative K-theory class defined by the pair $(  E,  \tilde{E}|_{\partial B(R/2)}  )$, under the Chern-Weil homomorphism
\[
c_2: K(B(\frac{R}{2}), \partial B(\frac{R}{2}))
\to H^4(   B(\frac{R}{2}), \partial B(\frac{R}{2})      ).
\]
\end{lem}

\begin{lem}
The class defined by the pair $(  E,  \tilde{E}|_{\partial B(R/2)}  )$ inside the relative K-theory $K(B(\frac{R}{2}), \partial {B(\frac{R}{2})})$ is
\[
( \dim V) \text{Ind } [  S_- \xrightarrow{\sum y_\mu \hat{c }_\mu }  S_+, 0, 0 ].
\]
\end{lem}

\begin{proof}
Using the exact sequence (\ref{coherentsheafpresentation}), this class is 
\[
\text{Ind } [H^2_{D_{A_t}}\otimes S_- \xrightarrow{ } H^2_{D_{A_t}}\otimes S_+, \tilde{E}|_{\partial B(R/2)}, \alpha_t    ],
\]
which by continuity is 
\[
\text{Ind } [ (H^2_{D_{A_\infty}}\oplus V)\otimes S_- \xrightarrow{ } (H^2_{D_{A_\infty}}\oplus V) \otimes S_+, \tilde{E}|_{\partial B(R/2)}, \alpha_\infty ] 
\]
By the additivity of the index bundle construction, this is
\[
\text{Ind } [ H^2_{D_{A_\infty}}\otimes S_- \xrightarrow{ } H^2_{D_{A_\infty}} \otimes S_+, \tilde{E}|_{\partial B(R/2)}, \alpha_\infty ] + 
\text{Ind } [  V\otimes S_- \xrightarrow{ }  V \otimes S_+, 0, 0 ]
\]
Now by the analogue of the exact sequence (\ref{coherentsheafpresentation}) applied to $A_\infty$,
the first summand is the relative K-theory class defined by the pair $(\tilde{E}, \tilde{E}|_{\partial B(R/2)})
$, which is trivial. So the class simplifies to
$
( \dim V) \text{Ind } [  S_- \xrightarrow{\sum y_\mu \hat{c }_\mu }  S_+, 0, 0 ].
$
\end{proof}

\begin{lem}
The $c_2$ of the relative K-theory class $\text{Ind } [  S_- \xrightarrow{\sum y_\mu \hat{c }_\mu }  S_+, 0, 0 ]$ is 1.
\end{lem}

\begin{proof}
Using the construction of the index bundle,
\[
\begin{split}
\text{Ind }[  S_- \xrightarrow{\sum y_\mu \hat{c }_\mu }  S_+, 0, 0 ] 
= \text{Ind }[  \C^2\oplus S_- \xrightarrow{0\oplus \sum y_\mu \hat{c }_\mu }  S_+, \C^2|_{ \partial B(R/2)  }, Id_{\C^2}\oplus 0 ].
\end{split} 
\]
To make sense of this, we need to apply a perturbation to the map 
\[
\C^2\oplus S_- \xrightarrow{0\oplus \sum y_\mu \hat{c}_\mu  }  S_+
\]
to make it surjective. A particular choice is given by the map appearing in the ADHM construction of the standard 1-instanton (\cf \cite{DonaldsonKronheimer}, Section 3.4.1). Thus the relative K-theory class is given by the class of the 1-instanton, together with a trivialisation near infinity. This has charge $c_2=1$.
\end{proof}

Combining these Lemmas give Theorem \ref{instantonnumberequalsdimsmallspectrum}.


\begin{thebibliography}{7}
	
	
	
	
	
	\bibitem{Atiyah} M. F. Atiyah, \emph{K-theory}, New York : W. A. Benjamin, 1967 
	
	
	\bibitem{Bartocci1}	C. Bartocci, U. Bruzzo, D. Hern\'ahdez Ruip\'erez, 
	\emph{A hyperkghler Fourier transform}, Differential Geometry and its Applications 8 (1998) 239--249
	
	
	\bibitem{Bottcher} A. Böttcher, B. Silbermann,  \emph{Analysis of Toeplitz Operators},  Springer Monographs in Mathematics (2nd ed.), Springer-Verlag (2006)
	
	\bibitem{BraamandBaal} Peter J. Braam, Pierre van Baal, \emph{Nahm's Transformation for Instantons},
	Commun. Math. Phys. 122, 267-280 (1989)
	
	\bibitem{Cherkis} S.A. Cherkis, J. Hurtubise,
	\emph{Monads for Instantons and Bows},  arXiv:1709.00145  
	
	\bibitem{Christ} N. H. Christ, \emph{Self dual Yang Mills solutions}, Complex manifold techniques in theoretical physics, 45-54, Pitman Advanced Publishing Program
	
	\bibitem{AlainConnes} A. Connes, \emph{Geometry and the Quantum},  arXiv:1703.02470 
	
	
	
	\bibitem{DonaldsonRungeapproximation} S. K. Donaldson, \emph{The approximation of instantons}, Geometric and Functional Analysis
	Vol. 3, No. 2 (1993)
	
	
	
	\bibitem{DonaldsonKronheimer} S. K. Donaldson and P. B. Kronheimer, \emph{The Geometry of 4-manifolds}, Oxford Mathematical Monograghs, Oxford Science publications
	
	\bibitem{DonaldsonSegal} S. K. Donaldson and E. Segal, \emph{Gauge theory in higher dimensions, II}, arXiv:0902.3239, 18 Feb 2009
	
	
	\bibitem{DonaldsonSun2} S. K.  Donaldson and S. Sun, 
	 \emph{Gromov-Hausdorff limits of K¨ahler manifolds and
	algebraic geometry, II},  arXiv:1507.05082 


    \bibitem{Jardim} M. Jardim, \emph{A survey on Nahm transform}, Journal of Geometry and Physics 52 (2004) 313–327



	\bibitem{Kodaira} K. Kodaira,  \emph{Complex Manifolds and Deformation of Complex Structures}, Springer
	
%	\bibitem{Li1} Y. Li, \emph{An invariant approach to gauge theory via coupled Dirac equation},  PhD first year minor project           
	
    
	
	
	\bibitem{Witten} E. Witten, \emph{Some comments on the recent twistor space construction}, Complex manifold techniques in theoretical physics, 207-218, Pitman Advanced Publishing Program

	
		
		
		
		
	

	
\end{thebibliography}
\end{document}